\documentclass[10pt]{article}
\usepackage{amssymb,amsmath,amsfonts}
\usepackage{dsfont}
\usepackage{amsthm}
\usepackage{graphicx,color,subfig}
\usepackage{xspace}
\usepackage{mathrsfs}
\usepackage{amscd}
\usepackage{epsfig,color,graphics}

\newcommand{\T}{\ensuremath{\mathcal T}}

\let \Re \relax
\DeclareMathOperator{\Re}{Re}
\let \Im \relax
\DeclareMathOperator{\Im}{Im}
\newcommand{\Con}{\ensuremath{\mathscr C}}
\newcommand{\Cinf}{\ensuremath{\Con^\infty}}
\renewcommand{\S}{\ensuremath{\mathscr S}}

\DeclareMathOperator{\supp}{supp}

\newcommand{\hf}{\frac{1}{2}}

\newcommand{\est}[1]{\langle #1 \rangle}

\newcommand{\mb}[1]{\ensuremath{\mathbb{#1}}}
\newcommand{\N}{{\mb{N}}}

\newcommand{\R}{{\mb{R}}}
\newcommand{\C}{{\mb{C}}}

\newcommand{\al}{\alpha}
\newcommand{\be}{\beta}
\newcommand{\de}{\delta}
\newcommand{\eps}{\varepsilon}

\renewcommand{\d}{\ensuremath{\partial}}

\DeclareMathOperator{\op}{op}
\DeclareMathOperator{\Op}{Op}

\newcommand{\inp}[2]{\langle #1, #2 \rangle}

\def\iint{\intop\!\!\intop}

\newcommand{\nhd}{neighborhood\xspace}

\usepackage[active]{srcltx}

\newcommand{\este}[1]{\langle #1 \rangle_\eps}
\def\VP{\varphi}

\newcommand{\tendvers}[2]{\xrightarrow[#1\rightarrow #2]{}}

\usepackage[french,english]{babel}

\newtheorem{lemma}{Lemma}[section]
\newtheorem{theorem}{Theorem}
\newtheorem{proposition}[lemma]{Proposition}
\newtheorem{corollary}[lemma]{Corollary}

\theoremstyle{definition}

\newtheorem{remark}{Remark}

\makeatletter
\def\keywords{
    \vspace{1ex}
    \noindent
    \if@twocolumn
      \small{\bf  Keywords}\/---$\!$    \else
      \begin{center}\small\ {\bf Keywords}\end{center}\quotation\small
    \fi}
\def\endkeywords{\vspace{0.6em}\par\if@twocolumn\else\endquotation\fi
    \normalsize\rm}
\makeatother

\begin{document}
\title{Carleman estimates for the Zaremba Boundary Condition and Stabilization of Waves.}

\author{Pierre Cornilleau\thanks{Teacher at Lyc\'ee du Parc des Loges, 1, boulevard des Champs-\'Elys\'ees, 91012 \'Evry, France.\newline e-mail: pierre.cornilleau@ens-lyon.org.}  \& Luc Robbiano\thanks{Laboratoire de Math\'ematiques de Versailles, Universit\'e de Versailles St Quentin,
CNRS, 45, Avenue des Etats-Unis, 78035 Versailles, France. e-mail : Luc.Robbiano@uvsq.fr}}

\maketitle
\begin{abstract}
In this paper, we shall prove a Carleman estimate for the so-called Zaremba problem.
Using some techniques of interpolation and spectral estimates, we deduce a result of stabilization for the wave equation by means 
of a linear Neumann feedback on the boundary. This extends previous results from the literature: 
indeed, our logarithmic decay result is obtained
while the part where the feedback is applied contacts the boundary zone driven by an homogeneous Dirichlet condition.
We also derive a controllability result for the heat equation with the Zaremba boundary condition.
\end{abstract}


\tableofcontents

%
%

\section{Introduction}

\subsection{General background}

We are interested here in the stabilization of the wave equation on a bounded
connected regular open set of $\R^d$. Our stabilization will be obtained by means of a
feedback on a part of the boundary while the other part of the boundary is
submitted to an homogeneous Dirichlet condition.

\noindent Since the works of Bardos, Lebeau, Rauch (see \cite{BLR1}), the case of
stabilization for the wave equation is well understood (by the so called
{\em{Geometric Control Condition}}) for Dirichlet or Neumann boundary condition. Indeed, if the part of the
boundary driven by the homogeneous Dirichlet condition does not contact the
region where the feedback is applied, Lebeau has given a sharp sufficient condition for exponential stabilization
 of the wave equation (see \cite[Th\'eor\`eme 3]{Leb} and \cite{Lebeau:96}). 
Moreover, Lebeau and Robbiano (see \cite{LR2}) have shown that, in the case where the Neumann boundary condition is applied on the
entire boundary, a weak condition on the feedback (which does not satisfy {\em{Geometric Control Condition}}) provides logarithmic decay of regular solutions. 
\newline \noindent On the other hand, multiplier
techniques (see \cite{KZ,CLO}) give some results of exponential stabilization (even if the part of the
boundary driven by the homogeneous Dirichlet condition touches the region where the feedback is applied)
but under very strong assumptions on the form of the boundary conditions.

\noindent Our goal here is to obtain some stabilization of logarithmic type under weak
assumptions for the boundary conditions. More precisely, we will see that, for solutions driven by an homogeneous Dirichlet boundary condition on 
a part of the boundary and submitted to a feeback of the form $$\partial_{\nu} u = - a (x) \partial_t u$$ on the other part of the boundary,
where $a$ is some non-trivial non-negative function, their energy with initial data in the domain of $\mathcal{\mathcal{A} }^k$
(denoting $\mathcal{\mathcal{A} }$ the infinitesimal generator of our evolution equation) decays like $\ln (t)^{-
k}$ when $t$ goes to infinity.

To this end, we will need some Carleman estimates for the so-called Zaremba Boundary Problem
\begin{equation*}
  \left\{ \begin{array}{l}
\Delta_X u = f\\
    u = f_0\\
    \partial_{\nu} u = f_1\\
  \end{array} \begin{array}{l}
    \text{ in } X, \\
    \text{ on } \partial X_D,\\
    \text{ on } \partial X_N,\\
  \end{array} \right.
\end{equation*}
where $X$ is some regular manifold with boundary $\partial X$ splitted into $\partial X_D$ and $\partial X_N$ and normal vectorfield $\nu$. 
However, we will mainly tackle some local problem and the following model case (in $\R^n_+$ with the flat metric)
\begin{equation*}
  \left\{ \begin{array}{l}
\Delta u = f\\
    u = f_0\\
    \partial_{x_n} u = f_1\\
  \end{array} \begin{array}{l}
    \text{ in } \{x_n>0\}, \\
    \text{ on } \{x_n=0, x_1>0\},\\
    \text{ on } \{x_n=0, x_1<0\},\\
  \end{array} \right.
\end{equation*}
 should help the reader to understand the main difficulties of this problem.
\newline \noindent The Zaremba problem lies in the large class of boundary pseudodifferential operators, studied by many authors.
The first one was probably Eskin (see the monograph \cite{Esk} where pseudodifferential elliptic boundary problems are studied) but then Boutet de Monvel
- in \cite{BdeM} - raised the fundamental {\it transmission condition}. It was shown to play a key role in the resolution of such problems
(see the books of Grubb \cite{Grubb} and \cite[Chapter 10]{Grubb2} where the algebra of pseudodifferential problems is studied in details). 
\newline \noindent Unfortunately, the Zaremba problem can not be solved by this pseudodifferential calculus. Indeed, its resolution involves 
a pseudodifferential operator on the boundary that does not satisfy the transmission condition (see \cite{HS}).
It lies in the general class of operators introduced by Rempel and Schulze in \cite{RS} which allow to construct a parametrix for mixed elliptic problems
- including the Zaremba problem (see \cite{HS} and, more specifically, Section 4.1). 
However, up our knowledge, a Carleman estimate for the Zaremba problem could not be obtained so far. We should nonetheless mention Fu's work \cite{Fu} where the author proves a global Carleman inequality in a Laplace problem with mixed - but different - boundary conditions.
\newline Carleman estimates have many applications ranging from the quantification of unique continuation problems, inverse problems, to stabilization issues and control theory
(see the survey paper \cite{LL} for a general presentation of these topics). This last
application was the motivation for the proof of a suitable Carleman estimate (in the papers of either
Lebeau and Robbiano \cite{LR1} or Fursikov and Imanuvilov \cite{FI}) and is still animating nowadays a large developpement of Carleman estimates (see e.g. \cite{LRR1,LRL} were controllability of parabolic systems with non-smooth coefficients is studied).
Finally, we use the approach developped in \cite{Lebeau:96,LR2,Bu} (also used by other authors - see, e.g., \cite{Bel}) to deduce 
our stabilization result. We shall also address a controllability result for the heat equation with the Zaremba boundary condition
(based on the approach developped in \cite{LR1}).

\vspace{0.3 cm} {\bf Acknowledgements.} The authors thank J\'er\^ome Le Rousseau and Nicolas Lerner for interesting discussions related to this work.

\subsection{Stabilization of Waves}

Let $\Omega$ be a bounded connected open set of $\mathbb{R}^d$ with
$\Cinf$ boundary $\partial \Omega$. Let also $\Gamma$ a smooth hypersurface of $\partial \Omega$ which splits the boundary into the 
two non-empty open sets  $\partial \Omega_D$, $\partial\Omega_N$ so that $\partial \Omega = \partial \Omega_D \sqcup
\partial \Omega_N \sqcup \Gamma$ (see Figure 1). 
\newline \noindent We study the decay of the solution of the following
problem
\begin{equation}\label{ondesmelees}
  \left\{ \begin{array}{l}
    (\partial_t^2 - \Delta) u = 0\\
    u = 0\\
    \partial_{\nu} u + a (x) \partial_t u = 0\\
    (u, \partial_t u)_{|t = 0} = (u_0, u_1)
  \end{array} \begin{array}{l}
    \text{ in } \Omega \times \mathbb{R}^+,\\
    \text{ on } \partial \Omega_D \times \mathbb{R}^+,\\
    \text{ on } \partial \Omega_N \times \mathbb{R}^+,\\
    \text{ in } \Omega,
  \end{array} \right.
\end{equation}
where $(u_0, u_1) \in H^1 (\Omega) \times L^2 (\Omega)$ is such that $u_0 = 0$ in
$\partial \Omega_D$ and $a$ is, for some $\rho\in(0,1)$, a non-negative function of $\Con^{\rho}(\partial \Omega_N)$,
the space of H\"{o}lder continuous functions on $\partial \Omega_N$.

\par For the sake of simplicity, we here focus on the classical Laplacian $\Delta$ but all the results described below remain true with
the Laplacian associated to a smooth metric (see Section 4).

\begin{figure}[h]
\centerline{\begin{picture}(0,0)%
\includegraphics{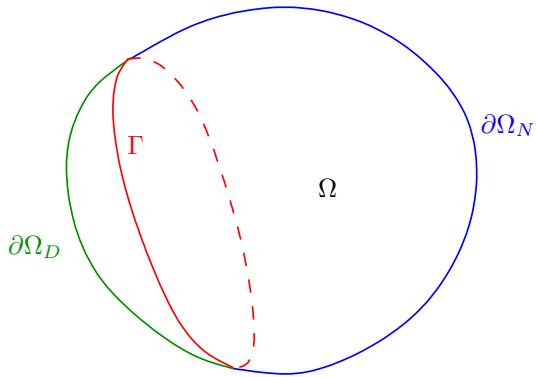}%
\end{picture}%
\setlength{\unitlength}{5594sp}%
\begingroup\makeatletter\ifx\SetFigFont\undefined%
\gdef\SetFigFont#1#2#3#4#5{%
  \reset@font\fontsize{#1}{#2pt}%
  \fontfamily{#3}\fontseries{#4}\fontshape{#5}%
  \selectfont}%
\fi\endgroup%
\begin{picture}(2126,1650)(2132,-1911)
\put(2147,-1368){\makebox(0,0)[lb]{\smash{{\color[rgb]{0,.56,0}$\partial \Omega_D $}}}}
\put(4243,-829){\makebox(0,0)[lb]{\smash{{\color[rgb]{0,0,1}$\partial \Omega_N $}}}}
\put(2677,-922){\makebox(0,0)[lb]{\smash{{\color[rgb]{1,0,0}$\Gamma $}}}}
\put(3521,-1116){\makebox(0,0)[lb]{\smash{{\color[rgb]{0,0,0}$\Omega $}}}}
\end{picture}%
}
\caption{A configuration example.}
\end{figure}
\noindent We denote $H =\{u_0 \in H^1 (\Omega) ; u_0 = 0 \text{ in } \partial \Omega_D \}
\times L^2 (\Omega)$ and define
\[ \mathcal{\mathcal{A} }= \left( \begin{array}{l}
     0\\
     \Delta
   \end{array} \begin{array}{l}
     I\\
     0
   \end{array} \right) \]
with domain
\[ \mathcal{D}(\mathcal{\mathcal{A} }) =\{(u_0, u_1)\in H ; \Delta u_0 \in L^2 (\Omega), u_1
   \in H^1 (\Omega), u_0 = 0 \text{ on } \partial \Omega_D \text{ and }
   \partial_{\nu} u_0 + a (x) u_1 = 0 \text{ on } \partial \Omega_N \}. \]
For any solution $u$ of \eqref{ondesmelees}, we define its energy by
\[ E (u, t) = \frac{1}{2} \int_\Omega | \partial_t u (x, t) |^2 + | \partial_{x} u (x, t)
   |^2 d x \]
where $\partial_x=(\partial_{x_1},...,\partial_{x_d})$.
\newline \noindent Denoting the resolvent set of $\mathcal{A}$ by
$$\rho (\mathcal{A})=\{ \mu \in \mathbb{C}; \mathcal{A}-\mu I: \mathcal{D(A)} \rightarrow  H \text{ is an isomorphism} \},$$
we will establish the following spectral estimate:

\begin{proposition}\label{spectral}
Let $\rho>1/2$ and $a\in\Con^{\rho}(\partial \Omega_N)$ a non-negative function.\\
If $a\neq 0$ and $ a(x) \tendvers{x}{\Gamma} 0$, then one has $i \mathbb{R} \subset \rho (\mathcal{A} )$ and there exists $C>0$ such that
  \[ \forall \lambda \in \mathbb{R}, \ \| (\mathcal{\mathcal{A} }- i \lambda I)^{- 1} \|_{H \rightarrow H}
     \leqslant C e^{C|\lambda |} . \]
\end{proposition}
Hence, using an
useful result of Burq (see \cite[Theorem 3]{Bu}), we get our logarithmic decay result:
\begin{theorem}\label{decroissancelog}
Let $\rho>1/2$ and $a\in\Con^{\rho}(\partial \Omega_N)$ a non-negative function.\\
If $ a(x) \tendvers{x}{\Gamma} 0$ and $a \neq 0$ then, for every $k \ge 1$, there exists $C_{k}> 0$ such that, for every $(u_{0,} u_1) \in
  \mathcal{D}(\mathcal{\mathcal{A} }^k)$ the corresponding solution $u$ of \eqref{ondesmelees}
  satisfies
  \[ \forall t \geqslant 0, \ E (u, t)^{1 / 2} \leqslant
     \frac{C_{k}}{\log (2 + t)^k} \left\| (u_0, u_1)
     \right\|_{\mathcal{D}(\mathcal{\mathcal{A} }^k)} . \]
\end{theorem}
These results are completely analogous to the ones obtained by Lebeau and
Robbiano in \cite{LR2}. The outline of the proof is also quite similar to the one
proposed there except that the situation is a bit different here because of the
mixed character of the boundary value problem. 
\newline The key point is also to establish some Carleman estimate in a \nhd of $\Gamma$ and to obtain some interpolation inequality (see \cite[Th\'eor\`eme 3]{LR2}).
This last result concerns an abstract problem derived from the spectral problem.
\newline Defining $X = (- 1, 1) \times \Omega, \partial X_N = (- 1, 1) \times \partial \Omega_N,
\partial X_D = (- 1, 1) \times \partial \Omega_D$, we consider the corresponding problem:
\begin{equation}\label{ondesabstrait}
  \left\{ \begin{array}{l}
    \Delta_X v = v_0\\
    (\partial_{\nu} + i a (x) \partial_{x_0}) v = v_1\\
    v = 0
  \end{array} \begin{array}{l}
    \text{ in } X,\\
    \text{ on } \partial X_N,\\
    \text{ on } \partial X_D,
  \end{array} \right.
\end{equation}
for some data $v_0 \in L^2(X) $ and $v_1\in L^2(\partial X_N)$.

\noindent If $Y = (- 1/ 2, 1 / 2) \times \Omega$ and $\partial X^{\delta}_N=(-1,1)\times \{x\in \partial \Omega_N; a(x)>\delta\}$, 
we will prove the following interpolation result.

\begin{proposition}\label{interpolation} Let $\rho>1/2$ and $a\in\Con^{\rho}(\partial \Omega_N)$ a non-negative function.\\
  If  $ a(x) \tendvers{x}{\Gamma} 0$ and $a \neq 0$, there exists $\delta>0$, $C > 0$ and $\tau_0 \in (0, 1)$ such that for
  any $\tau \in [0, \tau_0]$ and for any function $v$ solution of \eqref{ondesabstrait}, the
  following inequality holds
  \[  \| v \|_{H^1 (Y)} \leqslant C \left( \| v_0\|_{L^2 (X)} + \| v_1 \|_{L^2 (\partial X_N)} 
+\|v \|_{L^2 (\partial X^{\delta}_N)}+ \|\partial_{x_0} v \|_{L^2 (\partial X^{\delta}_N)} \right)^{\tau} \| v\|^{1 - \tau}_{H^1 (X)}. \]
\end{proposition}

\subsection{Carleman estimates for the Zaremba Boundary Condition}
We will now present our Carleman estimates and establish first some useful notations. Let $n \ge 2$ be the dimension of the connected manifold $X$.
\subsubsection{Notations}\label{Notations}
\paragraph{Pseudodifferential operators}
We use the notation introduced in  \cite{LR1}.

\noindent First, we shall use in the sequel the notations $\est{\xi}:= (1+|\xi|^2)^\hf$ and $D_{x_j}=\frac{h}{i}\partial_{x_j}$ for $1 \le j\le n$.

\noindent Let us now introduce semi-classical $\psi$DOs. We denote by
$S^m(\R^{n}\times\R^{n})$, $S^m$ for short, the space of smooth
functions $a(x,\xi,h)$, defined for $h \in (0,h_0]$ for some
$h_0>0$, that satisfy the following property: for all $\al$, $\be$
multi-indices, there exists $C_{\al,\be}\geq 0$, such that
\begin{align*}
  \left|\d_{x}^\al \d_{\xi}^\be a (x,\xi,h) \right|\leq
  C_{\al,\be} \est{\xi}^{m-|\be|}, \quad x \in \R^{n}, \
  \xi \in \R^{n},\ h \in (0,h_0]. 
\end{align*}
Then, for all sequences $a_{m-j} \in S^{m-j}$, $j\in \N$, there exists
a symbol $a \in S^m$ such that $a \sim \sum_j h^j a_{m-j}$, in the
sense that $a - \sum_{j<N} h^j a_{m-j} \in h^N S^{m-N}$ (see for
instance \cite[Proposition 2.3.2]{Martinez:02} or \cite[Proposition
18.1.3]{Hoermander:V3}), with $a_m$ as principal symbol.  We define
$\Psi^m$ as the space of $\psi$DOs $\mathcal{A}  = \Op(a)$, for $a \in S^m$,
formally defined by
\begin{align*}
  \mathcal{A} \, u (x) = \frac{1}{(2\pi h)^{n}} \iint e^{i\inp{x-t}{\xi}/h} 
  a(x,\xi,h)\: u(t)\  d t \  d \xi, 
  \quad u \in \S'(\R^{n}).
\end{align*}

We now introduce tangential symbols and associated operators.  
\newline \noindent We set
$x=(x',x_n)$, $x'=(x_1,\dots,x_{n-1})$ and
$\xi'=(\xi_1,\dots,\xi_{n-1})$ accordingly. We denote by
$S^m_\T(\R^{n}\times\R^{n-1})$, $S^m_\T$ for short, the space of
smooth functions $b(x,\xi',h)$, defined for $h \in (0,h_0]$ for some
$h_0>0$, that satisfy the following property: for all $\al$, $\be$
multi-indices, there exists $C_{\al,\be}\geq 0$, such that

 \begin{align*}
   \left| \d_{x}^\al \d_{\xi'}^\be b (x,\xi',h) \right| \leq
  C_{\al,\be} \est{\xi'}^{m-|\be|}, \quad x\in \R^{n}, \
   \xi' \in \R^{n-1},\ h \in (0,h_0]. 
\end{align*}
As above, for all sequences $b_{m-j} \in S^{m-j}_\T$, $j\in \N$, there
exists a symbol $b \in S^m_\T$ such that $b \sim \sum_j h^j
b_{m-j}$, in the sense that $b - \sum_{j<N} h^j b_{m-j} \in h^N S^{m-N}_\T$,
with $b_m$ as principal symbol. We define $\Psi^m_\T$ as the space of
tangential $\psi$DOs $B = \op(b)$ (observe the notation we adopt is different from above to avoid confusion), for $b \in S^m_\T$, formally defined by
\begin{align*}
  B\, u (x) = \frac{1}{(2\pi h)^{n-1}} \iint e^{i\inp{x'-t'}{\xi'}/h } 
  b(x,\xi',h)\: u(t',x_n)\  d t' \  d \xi', 
  \quad u \in \S'(\R^{n}).
\end{align*}
We shall also denote the principal symbol $b_m$ by $\sigma(B)$.

\paragraph{Different norms}
We use $L^2$ and $H^s_{sc}$ semi-classical norms on $\R^n$, on $\{x_n>0\}$, on $\{x_n=0\}$ and on $\{x_n=0,\ \pm x_1 > 0\}$. 
We recall that, in this paper, we use the usual semi-classical notations, namely $D_{x_j}=\frac{h}{i}\partial_{x_j}$, and the symbols are quantified in semi-classical sense. In particular all the norms depend on $h$.

To distinguish 
these different norms, we denote by 
$$
\|u\|^2=\int_{\R^n}|u(x)|^2dx, \quad   \|u\|_{s}=\|\Op(\est{\xi}^s)u\|
$$
and
$$
\|u\|^2_{L^2(x_n>0)}=\int_{\{x\in\R^n,\ x_n>0\}}|u(x)|^2 d x, \quad 
 \|u\|^2_{H^1_{sc}(x_n>0)}= \|u\|^2_{L^2(x_n>0)}+\sum_{j=1}^{n}\|D_{x_j}u\|^2_{L^2(x_n>0)}.
$$
Finally, on $x_n=0$, we use the norms
$$
|v|^2=\int_{\R^{n-1}}|v(x')|^2dx',\quad |v|_{s}=|\op(\est{\xi'}^s)v|,
$$
and the space $H^s_{sc}(\pm x_1 > 0)$ of the restrictions of $H^s_{sc}(\R^{n-1})$ functions equipped with the norm
$$
|v|_{H^s_{sc}(\pm x_1 > 0)}
=\inf_{\genfrac{}{}{0pt}{}{w\in H^s_{sc}(\R^{n-1})}{w_{|\pm x_1 > 0}=v}}   |\op(\est{\xi'}^s)w| .
$$
In particular for $v\in H^s(\R^{n-1})$, we have 
$$
|v_{|\pm x_1 > 0}|_{H^s_{sc}(\pm x_1 > 0)}\le |v|_{s}
$$
and we write, when there is no ambiguity, $ |v|_{H^s_{sc}(\pm x_1 > 0)}$ instead of $|v_{|\pm x_1 > 0}|_{H^s_{sc}(\pm x_1 > 0)}$.

\subsubsection{Carleman estimate}
We now detail the local Carleman estimate obtained for the Zaremba boundary problem. 
\newline
\par Let $B_\kappa =\{x \in \mathbb{R}^n ; |x| \leqslant \kappa \}$ 
and $P$ a differential operator whose form is
$$ P=-\partial_{x_n}^2+R\left(x,\frac{1}{i}\partial_{x'}\right)$$
where $\partial_{x'}= (\partial_{x_1}, \ldots, \partial_{x_{n - 1}})$ and the symbol $r(x,\xi')$ of $R$ is real, homogeneous of degree 2 in $\xi'$ and satisfies
$$\left\{ \begin{array}{l} \exists c>0; \forall (x,\xi')\in B_\kappa\times \R^{n-1}, \ r(x,\xi')\ge c |\xi'|^2,\\
   \\
\forall \xi'\in \R^{n-1}, r(0,\xi')=|\xi'|^2.
  \end{array}\right. 
$$
As usual in the context of Carleman estimates, we define the conjugate $P_{\varphi} = h^2 e^{\varphi / h} \circ P \circ e^{- \varphi /
h}$ for $\varphi$ any real-valued $\Cinf$ function. Since
$$P_{\varphi}=h^2\left(\frac{1}{i}\partial_{x_n}+\frac{i}{h} \partial_{x_n} \varphi\right)^2+h^2 R\left(x,\frac{1}{i}\partial_{x'} +\frac{i}{h}\partial_{x'} \varphi \right)  ,$$
the corresponding semi-classical principal symbol satisfies
\[ p_{\varphi} (x, \xi) =  (\xi_{n} + i \partial_{x_n} \VP  (x))^2 + r (x, \xi' + i\partial_{x'} \varphi (x)). \]
 We assume that $\varphi$ is such that, for some $\kappa_0>0$,
\begin{equation}\label{nonzero}
  \forall x \in B_{\kappa_0}, \frac{\partial \varphi}{\partial x_n} (x) \neq 0
\end{equation}
and that H\"ormander pseudo-convexity hypothesis (see \cite[Paragraph 28.2,
28.3]{Hoermander:V3}) holds for $P$ on $B_{\kappa_0}$
\begin{equation}\label{hypohormander}
  \forall (x, \xi) \in B_{\kappa_0} \times \mathbb{R}^n, p_{\varphi} (x, \xi) = 0
  \Rightarrow \{\Re p_{\varphi}, \Im p_{\varphi} \}(x, \xi) > 0,
\end{equation}
where the usual Poisson bracket is defined, for $p, q$ smooth functions, by
  \[ \{p , q\} (x, \xi) = (\partial_{\xi} p. \partial_x q - \partial_x p.
     \partial_{\xi} q) (x, \xi). \]
\begin{remark}
 For instance in the model case $P=-\Delta$, we can choose $\displaystyle \VP (x_n)= x_n+\frac{a}{2}x_n^2$. We have indeed $p_\VP(x,\xi)=(\xi_n+i(1+ax_n))^2+|\xi'|^2$
thus $\{\Re p_\VP, \Im p_\VP\}(x,\xi)=4a(\xi_n^2+(1+ax_n)^2)>0$ if $x_n$ is small enough. 

In the general case, changing  $\VP$ into $e^{\beta \VP}$ for $\beta>0$ large enough, hypothesis \eqref{hypohormander}  can be satisfied (see \cite[Proposition 28.3.3]{Hoermander:V3} or \cite[Proof of Lemma 3, page 352]{LR1}).
\end{remark}
Our local Carleman estimate for the Zaremba Boundary Condition can now be stated in the following form.
\begin{theorem}\label{Carleman-mixte}
There exists $\eps>0$ such that if $\varphi$ satisfies
$$  \left( \frac{\partial \varphi}{\partial x_n} > 0 \text{ on } \{x_n=0\} \cap B_{\kappa_0} \right)\text{ and }|\partial_{x'} \VP(0)|\le \eps \partial_{x_n} \VP(0)$$ and \eqref{nonzero}, \eqref{hypohormander} hold
then,  there exists $\kappa\in(0,\kappa_0]$ and $ C, h_0 > 0$, such that, for any $h \in (0,
  h_0)$, $g_0\in H^{1/2}(x_1>0)$, $g_1 \in H^{-1/2}(x_1<0)$ and any $g \in H^1 (\R^n)$ supported in $B_\kappa$ which satisfies 
$$ P(g)\in L^2(\R^n) \text{ and } 
 \left\{ \begin{array}{l}
       g = g_0\\
       \partial_{x_n} g = g_1
     \end{array} \begin{array}{l}
       \text{ if } x_n=0 \text{ and } x_1>0,\\
       \text{ if } x_n=0 \text{ and } x_1<0,
     \end{array} \right. $$
the following inequality holds:
\begin{eqnarray*}&\| g e^{\varphi / h}\|_{H^1_{sc}(x_n>0)} + | g e^{\varphi / h}|_{1/2}  + |h (\partial_{x_n} g) e^{\varphi / h}|_{-1/2}&\\
&\le C\left(h^{-1/2}\|h^2 P(g) e^{ \varphi / h}\|_{L^2(x_n>0)}+  | g_0 e^{ \varphi / h}|_{H^{1/2}_{sc}(x_1>0)} +|h g_1 e^{ \varphi / h}|_{H^{-1/2}_{sc}(x_1<0)}\right).&
\end{eqnarray*}
\end{theorem}

\vspace{0.5cm}
\begin{remark}
 The estimate in the theorem, except for the boundary terms, is the usual Carleman estimate. Let us also remind that all the norms are semi-classical: in particular $\| g e^{\varphi / h}\|_{H^1_{sc}(x_n>0)}$ is equivalent to $h\| e^{\varphi / h} \partial_x g\|_{L^2(x_n>0)}+\| e^{\varphi / h} g\|_{L^2(x_n>0)} $.
For the other norms, we refer the reader to the definitions in paragraph \ref{Notations}.
\end{remark}

\begin{remark}
The  norms  $|.|_{1/2}$ and $|.|_{-1/2}$ on the boundary $x_n=0$ of the left hand side of this inequality cannot be replaced by the norms $|.|_{1}$ and $|.|_{}$
( provided that the data $g_0$, $g_1$ are estimated in the spaces $H^{1}(x_1>0)$ and $L^2(x_1<0)$).\\
Indeed, in the special case where $P=-\Delta$,
it is well-known that the variational solution of the boundary value problem
\[ \left\{ \begin{array}{l}
       -\Delta u = f\\
       u = 0\\
      \partial_{\nu} u=0
     \end{array} \begin{array}{l}
       \text{ in } X,\\
       \text{ on } \partial X_D,\\
       \text{ on } \partial X_N,
     \end{array} \right. \]
may be, even for smooth data $f$, such that $\partial_{\nu} u\notin L^2(\partial X)$.
\\ We refer the reader to the famous two-dimensionnal conterexample of Shamir (see \cite{Sh}) where one consider, in polar coordinates, the sets
$$X=\{(r,\theta); r\in(0,1), \theta\in(0,\pi)\}, \quad \partial X_N=\{(r,\pi); r\in(0,1)\},\quad \partial X_D=\partial X\setminus\overline{\partial X_N}.$$ 
and the function
$$ u(r,\theta)=\phi(r) r^{1/2}\sin\left( \frac{\theta}{2}\right)$$
with $\phi\in \Cinf([0,1])$ a cut-off function such that $\phi=1$ in some \nhd of $0$ and $\supp(\phi)\subset [0,1)$.
\end{remark}

The paper is structured as follows: our proof of the main Carleman
estimate (Theorem \ref{Carleman-mixte}) is divided into the three subsections of Section 2. 
This will allow us to deduce the interpolation inequality of Proposition \ref{interpolation} and finally Theorem \ref{decroissancelog} in Section 3.
In Section 4, we conclude by some comments on the geometry and sketch a proof of controllability of the heat equation with the Zaremba boundary condition. 
%
%
\section{Proof of Theorem \ref{Carleman-mixte}}
We first recall some well-known facts about pseudodifferential operators. We refer the reader to \cite{Martinez:02}. For simplicity, we write in all this section $\|\cdot\|_{H^s(x_n>0)}$ instead of $\|\cdot\|_{H^s_{sc}(x_n>0)}$. Note that there will be no confusion as we do not use the classical norm on $H^s(x_n>0)$.
 \newline \par\noindent \textbf{Composition formula.} If $a \in S^m$, $b \in S^{m'}$ then $\Op (a) \circ \Op (b) = \Op (c)$ for $c \in S^{m +
m'}$ given by
$$
  c (x, \xi,h) = \left( \sum_{| \alpha | \leqslant N} \frac{(h / i)^{|
  \alpha |}}{\alpha !} \partial^{\alpha}_{\xi} a \partial^{\alpha}_x b \right)
  (x, \xi,h) + h^{N + 1} R (x, \xi,h)
$$
where
\[ R (x, \xi,h) = \frac{N + 1}{(2 \pi h)^n} \int_0^1 (1 - t)^N \sum_{|
   \alpha | = N + 1} \frac{1}{i^{| \alpha |} \alpha !} \int_{\mathbb{R}^{2 n}}
 e^{- i z. \zeta / h} \partial^{\alpha}_{\xi} a (x, \xi + \zeta,h)
   \partial^{\alpha}_x b (x + t z, \xi,h) d z d \zeta d t. \]
We will also use the composition formula for tangential operators, which is completely analogous.
\newline In the sequel, we will also need the following straightforward result.
\begin{lemma}\label{regles}
Let $a\in S^m_\T$. Then
 $$[D_{x_n},\op(a)]=\frac{h}{i}\op(\partial_{x_n} a).$$
\end{lemma}

Next, we use the same notations as in \cite{LR2} and put \begin{eqnarray}
& p_\VP(x,\xi)=\xi_n^2+2i(\partial_{x_n}\VP)\xi_n+q_2(x,\xi')+2iq_1(x,\xi')\label{pphi}\\
&\textrm{where }  q_2(x,\xi')=-(\partial_{x_n}\VP)^2+r(x,\xi')-r(x,\partial_{x'}\VP(x)) \textrm{ and }q_1(x,\xi')=\tilde r(x,\partial_{x'}\VP(x),\xi').\nonumber
\end{eqnarray}
Here we denote by $\tilde r(x,.,.)$ the bilinear form associated to $r(x,.)$ (i.e. such that $r(x,\xi')=\tilde r(x,\xi',\xi')$ for all $\xi'\in \mathbb{R}^{n-1}$).

\noindent We also define $$\mu(x,\xi'):=q_2(x,\xi')+\frac{q_1(x,\xi')^2}{(\partial_{x_n}\VP(x))^2}.$$ 
The sign of $\mu$ is of great importance to localize the roots of $p_{\VP}$ in $\xi_n$. 
We may explain this from the model case presented in the introduction. In this framework, one has $P=-\Delta$ and we may choose $\varphi=\varphi(x_n)$ (more precisely of the form $\varphi(x_n)=x_n+a x_n^2/2$ for some $a>0$)  so that
$$p_\varphi(x,\xi')=(\xi_{n} + i \partial_{x_n} \VP  (x))^2+|\xi'|^2,$$
$$  q_2(x,\xi')=-(\partial_{x_n}\VP(x))^2+|\xi'|^2, \quad q_1(x,\xi')=0 \ \text{ and } \ \mu(x,\xi')=|\xi'|^2-(\partial_{x_n}\VP(x))^2.$$
Moreover, the roots of $p_{\VP}$ in $\xi_n$ are given by
$$\rho_{1}(x, \xi')=  -i(\partial_{x_n}\VP(x) - |\xi'|),  \quad \rho_{2}(x, \xi')=  -i(\partial_{x_n}\VP(x) + |\xi'|)$$
and satisfy
$$ \mu(x, \xi')>0 \Rightarrow \Im(\rho_{1}(x, \xi')) >0 > \Im(\rho_{2}(x, \xi'))$$
whereas
$$ \mu(x, \xi')<0 \Rightarrow \Im(\rho_{1,2}(x, \xi')) <0 .$$
In the microlocal zone $\mu <0$, the operator $p_\VP$ is elliptic and since its roots in $\xi_n$ have negative imaginary part, one will be able to estimate directly the traces of $g$ in terms of the interior data $P(g)$.
On the contrary, in the microlocal zone $\mu>0$, even is $p_\VP$ is elliptic, only one of its root in $\xi_n$ has negative imaginary part and elliptic estimates would only get an equation on the traces of $g$. 
In our general framework,  we prove several analogous properties presented in Lemma \ref{Lemme sur les racines} (which are very close to the ones of \cite[Lemme 3]{LR2}) and the case $\mu>0$ will in fact be treated in section 2.2. 
\newline \par Our proof of Theorem \ref{Carleman-mixte} is consequently divided in two main parts. In the first one, we establish a microlocal Carleman
inequality concentrated where $\mu<0$ and, in the second one, we focus on the microlocal region $\mu>-(\partial_{x_n} \VP)^2$. We will finally gather the results of these two parts in a short concluding section.
\newline
\par \noindent \textbf{Notations:} In the sequel, we set, for $w$ a function defined on $\R^n$, 
\begin{align*}
\underline{w}=\left\lbrace 
\begin{array}l
w \\
0 
\end{array}\begin{array}{l}
       \text{ if } x_n>0,\\
       \text{ if } x_n<0.
     \end{array}
\right.
\end{align*}
We also denote, for $z\in \C/\R_-$ and $s\in \R$,
$$z^s=\exp\left(s\log(z)\right)$$
where $\log$ is defined as an holomorphic function on $\mathbb{C} \backslash
\mathbb{R}_-$. Moreover, we use the notation $\sqrt{z}:=z^{1/2}$.
\subsection{Estimates in zone $\mu<0$}

We remind that we have denoted $v=e^{\VP/h}g$. We also define the set
$${\cal E}_{\alpha}=\{(x,\xi')\in\R^n\times\R^{n-1},\ \mu(x,\xi') \le -\alpha (\partial_{x_n}\VP)^2\}$$
where $\alpha>0 $ is a  sufficiently small parameter to be fixed later.

The proof we give essentially follows that of Lemma~4 in \cite{LR2} and Proposition~2.2 in \cite{LRR1}.

Let $\chi_-$ supported in ${\cal E}_{2\alpha}$ and satisfying $\chi_-=1$ in a \nhd of ${\cal E}_{3\alpha}$. Obviously $\chi_-\in S^0_\T$ because $\chi_-=0$ when $|\xi'|$ is large enough. 
\noindent If $u=\op(\chi_-)v$, one has
\begin{align}\label{estimation second membre}
& P_\VP u=\op(\chi_-)P_\VP v+[P_\VP,\op(\chi_-)]v=f_1 \textrm{ where }\nonumber\\
& \|f_1\|_{L^2(x_n>0)} \le C\|P_\VP v\|_{L^2(x_n>0)}+Ch\|v\|_{H^1(x_n>0)}.
\end{align}
Denoting $\delta^{(j)} = (d / d x_n)^j \delta_{|x_n = 0}$, straightforward computation show that we
have
\begin{equation}\label{extension} 
P_{\varphi} \underline{u} = \underline{f_1} - h^2   \gamma_0 (u) \delta' - i h (\gamma_1 (u) + 2 i \partial_{x_n} \varphi (x',
   0) \gamma_0 (u)) \delta 
\end{equation}
where $\gamma_0(u):=u_{|x_n = 0^+}$ and $\gamma_1(u):=D_{x_n}u_{|x_n = 0^+}=-ih\partial_{x_n}u_{|x_n = 0^+}$ are the first semi-classical traces.
\newline We now construct a local parametrix for $P_{\varphi}$. 
\newline \noindent Let $\chi (x,\xi) \in S^0$ such that $\chi = 1$ for sufficiently large $|\xi|$ as well as in a \nhd of $\supp(\chi_-)$ with moreover
$$\supp( \chi) \cap p^{-1}_{\varphi} (\{0\})=\emptyset.$$
Note that it is indeed possible because the real null set $p^{- 1}_{\varphi} (\{0\})$ is bounded in $\xi$ and, using Lemma \ref{Lemme sur les racines}, 
the roots of $p_{\VP}$ in $\xi_n$ are not real. 

\noindent We define
\[ e_0 (x, \xi) = \frac{\chi(x, \xi)}{p_{\varphi} (x, \xi)} \in
   S^{- 2}. \]
One may find $e_1 \in S^{-3}$ such that  $E = \Op (e_0+he_1)$ satisfies, for some $R_2 \in S^{-2}$,
\[ E \circ P_{\varphi} = \Op( \chi) + h^2 R_2 . \]
Indeed, by symbolic calculus, one may verify that $e_1=\chi \frac{\partial_x p_{\VP}.\partial_\xi p_{\VP}}{p^3_{\VP}}$. In the sequel, we shall denote $e:=e_0+he_1$.
\newline \noindent We set the new quantities
\begin{equation}\label{definitionsw}
  w_1 := \gamma_0 (u),\ w_0 := \gamma_1 (u) + 2 i \partial_{x_n} \varphi (x',
  0) \gamma_0 (u)
\end{equation}
and we apply our parametrix $E$ to the equation \eqref{extension} which may be written now in the form 
$$P_{\varphi} \underline{u} = \underline{f_1}+\frac{h}{i}w_0\delta -h^2 w_1 \delta'.$$ 
One computes the action of $E$ on $w_0$ and $w_1$ and finds
$$E\left(\frac{h}{i} w_0 \delta \right)(x',x_n)=\frac{1}{(2\pi h)^{n-1}} \iint e^{i (x'-y').\xi'/h} \hat{t_0}(x_n,x',\xi')w_0(y') d y' d \xi', $$
$$E\left(-h^2 w_1 \delta' \right)(x',x_n)=\frac{1}{(2\pi h)^{n-1}} \iint e^{i (x'-y').\xi'/h} \hat{t_1}(x_n,x',\xi')w_1(y') d y' d \xi', $$
where
$$\hat{t_0}(x_n,x',\xi')=\frac{1}{2i\pi}\int_{\R} e^{ix_n \xi_n/h} e(x,\xi) d \xi_n,$$
$$\hat{t_1}(x_n,x',\xi')=\frac{1}{2i\pi}\int_{\R} e^{ix_n \xi_n/h} \xi_n e(x,\xi) d \xi_n.$$
We note that the integral defining $\hat{t_0}$ is absolutely converging but that the integral defining $\hat{t_1}$ has to be understood is the sense of the
oscillatory integrals (see for instance \cite[Section 7.8]{Hoermander:V3}).
\newline \noindent Using the fact that $e(x,\xi', \xi_n)$ is holomorphic for large $|\xi_n|$ and actually a rational function with respect to $\xi_n$, we can change the contour $\R$ into the contour  defined by $\gamma = [- C \est{\xi'}, C \est{\xi'}] \cup \{\xi_n \in \mathbb{C};
| \xi_n | = C \est{\xi'} , \Im(\xi_n)>0\}$  oriented counterclockwise where $C>0$ is chosen sufficiently large so that $\chi=1$ if $|\xi_n|\ge C \est{\xi'}$.

Doing so, we get
\[ \underline{u} = E ( \underline{f_1}) + T_0 w_0 + T_1 w_1 + r_1 \]
where 
\begin{equation}\label{definition reste1} r_1 = (I - \Op ( \chi)) \underline{u} + h^2 R_2\underline{u}\end{equation}
and, if $j = 0, 1$ and $x_n > 0$, the tangential operators $T_j$ of symbols
\begin{equation}\label{symboletangentiel}
 \hat{t_j} (x, \xi') = \frac{1}{2 i \pi} \int_{\gamma} e^{i x_n (\xi_n / h)} e(x', x_n, \xi', \xi_n) \xi^j_n d \xi_n, 
\end{equation}
\newline \noindent The symbols $1-\chi$ and $\chi_-$ are not in the same symbol class but it is known (see Lebeau-Robbiano \cite{LR1} and Le Rousseau-Robbiano \cite[Lemma 2.2]{LRR1}) that, 
since $\supp (1-\chi)\cap\supp \chi_-=\emptyset$, we have
$$(I-\Op(\chi)) \op(\chi_-) \in \bigcap_{N\in\mathbb{N}} h^N \Psi^{-N}.$$
Consequently, recalling \eqref{definition reste1}, one has the estimate
\begin{equation}\label{reste1}\| r_1 \|_2 \leqslant C h \|\underline{v}\| = C h \| v\|_{L^2(x_n>0)}\le C h \| v\|_{H^1(x_n>0)}.\end{equation}
\par We now choose $\chi_1(x,\xi')\in S^0_\T$  so that $\supp( \chi_-)\subset \{\chi_1=1\}$, $\chi_1$ 
is supported in $\mathcal{E}_{\alpha}$ and   $\chi=1$ in a \nhd of $\supp( \chi_1)$.
\newline \noindent We set $t_j=\hat{t_j} \chi_1$ for $j=0,1$ which allows us to get
\begin{equation}\label{resolutionu}\underline{u}=E(\underline{f_1})+\op(t_0)w_0+\op(t_1)w_1+r_1+r_2,\end{equation}
where 
$$
r_2=\op((1-\chi_1)\hat{t_0})w_0+\op((1-\chi_1)\hat{t_1})w_1.
$$
One now notes that $|p_{\varphi} (x, \xi) | \geqslant c \est{\xi}^2$ on
$\supp ( \chi)$. Consequently, one obtains
\begin{equation}\label{estimationelliptique}
 \| E ( \underline{f_1})\|_1 \leqslant C \| \underline{f_1} \| = C \| f_1\|_{L^2(x_n>0)} .
\end{equation}
Moreover, using \eqref{symboletangentiel}, one may obtain
$$\forall l\in \mathbb{N}, \alpha\in \mathbb{N}^{n-1}, \beta\in \mathbb{N}^{n-1}, \  |\partial_{x_n}^l \partial^{\alpha}_{x'} \partial^{\beta}_{\xi'} \hat{t_j}|\le C h^{-l}\est{\xi'}^{-1+j+l-|\beta|}.$$
Consequently, noting that $r_2$ does not involve derivations with respect to $x_n$ and that 
$\supp(1-\chi_1)\cap \supp(\chi_{-|x_n=0})=\emptyset$, one obtains
\begin{equation}\label{reste2}
  \| r_2 \|_1 \leqslant C h ( \| v \|_{H^{1}(x_n>0)}+ | D_{x_n}v_{|x_n=0} |_{-1/2})
\end{equation}
from the composition of tangential operators and using the following trace formula (see \cite[page 486]{LR2})
\[|\psi_{|x_n = 0} | \leqslant C h^{- 1 / 2} \| \psi\|_{H^1(x_n>0)} . \]
\par Regarding the two last terms, we use that $\mu(x,\xi')<0$  for $(x,\xi')\in \supp(\chi_1)$. Hence, by Lemma \ref{Lemme sur les racines}, $p_{\VP}(x,\xi',\xi_n)^{-1}$ is an holomorphic function of $\xi_n$ on $\{ \Im(\xi_n) \ge 0 \}$ for $(x,\xi')\in \supp(\chi_1)$. 
\newline \noindent Recalling the form of $e$, one consequently has, for $j=0,1$,
\begin{eqnarray*}
(t_j)(x,\xi')&=& \frac{1}{2i\pi}\chi_1(x,\xi')\left( \int_\gamma e^{ix_n\xi_n/h}\frac{\xi_n^j}{ p_{\VP} (x,\xi',\xi_n)} d\xi_n+h \int_\gamma e^{ix_n\xi_n/h}\frac{\xi_n^j (\partial_x p_{\VP}.\partial_\xi p_{\VP})(x,\xi',\xi_n)}{ p^3_{\VP} (x,\xi',\xi_n)} d\xi_n\right)\\
&=&0.
\end{eqnarray*}
\par We shall now address the traces terms. We take the first two traces at $x_n=0^+$ of \eqref{resolutionu} which consequently gives, for $j=0,1$,
$$\gamma_j(u)= \gamma_j(E(\underline{f_1}))+\gamma_j(r_1)+\gamma_j(r_2).$$
Summing up equations \eqref{reste1}, \eqref{estimationelliptique} and \eqref{reste2}, one now deduces by trace formula
$$
h^{1/2}  |\gamma_0(u)|_{1/2}+h^{1/2}|\gamma_1(u)|_{-1/2} \leq C (\| f_1\|_{L^2(x_n>0)} +  h \|v\|_{H^1(x_n>0)}+ h| D_{x_n}v_{|x_n=0} |_{-1/2}).
$$ 
Using \eqref{resolutionu} again, one may deduce 
$$\|u\|_{H^1(x_n>0)}+h^{1/2}  |\gamma_0(u)|_{1/2}+h^{1/2}|\gamma_1(u)|_{-1/2}\le C( \|f_1\|_{L^2(x_n>0)} + h\|v\|_{H^1(x_n>0)} + h| D_{x_n}v_{|x_n=0} |_{-1/2}).$$
We finally come back to the original unknowns.  One has 
$$\gamma_0(u)=\op(\chi_-) v_{|x_n=0} \ \text{ and } \ \op(\chi_-) D_{x_n} v_{|x_n=0}=\gamma_1(u)-[D_{x_n},\op(\chi_-)]v_{|x_n=0}$$
which, using Lemma \ref{regles} and \eqref{estimation second membre}, allows us to get
\begin{eqnarray}\label{traces elliptiques}
&\|\op(\chi_-)v\|_{H^1(x_n>0)}+h^{1/2}|\op(\chi_-) v_{|x_n=0}|_{1/2}+h^{1/2}|\op(\chi_-) D_{x_n} v_{|x_n=0}|_{-1/2}&\\
&\le C (\| P_{\VP} v\|_{L^2(x_n>0)} + h \|v\|_{H^1(x_n>0)}+ h| D_{x_n}v_{|x_n=0} |_{-1/2}).&\nonumber
\end{eqnarray}

\subsection{Estimates in zone $\mu>-(\partial_{x_n}\varphi)^2$}

We denote by $v=e^{\VP/h}g$ and
\begin{align}
 v_0&=v_{|x_n=0}-\left(  e^{\VP/h} \right) _{|x_n=0}g_0\in H^{1/2}(x_n=0),\nonumber\\
v_1&=(D_{x_n}v)_{|x_n=0}+i(\partial_{x_n}\VP)_{|x_n=0}v_0+\left(  e^{\VP/h} \right) _{|x_n=0}(ihg_1+i(\partial_{x_n}\VP)_{|x_n=0}g_0)\in H^{-1/2}(x_n=0).\label{eq 2 : mu grand}
\end{align}
We have $\supp v_0\subset\{x'\in\R^{n-1},\  x_1\ge0\}$ and $\supp v_1\subset\{x'\in\R^{n-1},\  x_1\le0\}$. We consider $v_0$ and $v_1$ as unknown in the problem and in the sequel the goal is to obtain an equation on $v_0$ and $v_1$.
\newline \noindent The boundary conditions take the following form
\begin{align}
 v_{|x_n=0}&=v_0+G_0,\nonumber\\
(D_{x_n} v)_{|x_n=0}&= v_1-i(\partial_{x_n}\VP)_{|x_n=0}v_0+G_1,\label{eq 3 : mu grand}
\end{align}
where, following \eqref{eq 2 : mu grand}, we have
\begin{align}
 |G_0|_{1/2}&\le|e^{\VP/h}g_0|_{1/2},\nonumber \\
|G_1|_{-1/2}&\le h|e^{\VP/h}g_1|_{-1/2}+C|e^{\VP/h}g_0|_{1/2}.\label{eq 4 : mu grand}
\end{align}

We remark that if $g_0$ is fixed on $x_1<0$ and $g_1$ is fixed on $x_1>0$ we can extend $g_0$ and $g_1$ on $\R^{n-1}$ such that 
$$|e^{\VP/h}g_0|_{H^{1/2}(x_1<0)}\le |e^{\VP/h}g_0|_{1/2}\le 2|e^{\VP/h}g_0|_{H^{1/2}(x_1<0)}$$ 
and  $$|e^{\VP/h}g_1|_{H^{-1/2}(x_1>0)}\le |e^{\VP/h}g_1|_{-1/2}\le 2|e^{\VP/h}g_1|_{H^{-1/2}(x_1>0)}.$$ These extensions depend on  $h$.

Let $${\cal F}_\alpha=\{((x,\xi')\in\R^n\times\R^{n-1},\ \mu(x,\xi')\ge -(1-\alpha)(\partial_{x_n}\VP)^2\}$$ where $\alpha $ is small enough.
\newline \noindent Let $\chi_+(x,\xi')$ supported in ${\cal F}_{2\alpha}$ and satisfying $\chi_+=1$ in a \nhd of ${\cal F}_{3\alpha}$. Obviously $\chi_+\in S^0_\T$ because $\chi_+=1$ when $|\xi'|$ is large enough. 
\newline \noindent Let $u=\op(\chi_+)v$. We have
\begin{align}
& P_\VP u=\op(\chi_+)P_\VP v+[P_\VP,\op(\chi_+)]v=f_1 \textrm{ where }\nonumber\\
& \|f_1\|_{L^2(x_n>0)}\le C\|P_\VP v\|_{L^2(x_n>0)}+Ch\|v\|_{H^1(x_n>0)}.\label{eq 6, mu grand}
\end{align}
Let $\chi_1$ supported in ${\cal F}_{\alpha}$ such that $\chi_1=1$ on a \nhd  of ${\cal F}_{2\alpha}$, in particular on  $\supp\chi_+$.
When the roots $\rho_j$ are well defined (see the Lemma \ref{Lemme sur les racines}), we have by \eqref{pphi}  $\rho_1+\rho_2=-2i(\partial_{x_n}\VP)  $ and $\rho_1\rho_2=q_2+2iq_1  $. 
Hence, we obtain
\begin{align}
 (D_{x_n}-\op(\rho_2\chi_1))(D_{x_n}-\op(\rho_1\chi_1))u&=D_{x_n}^2u-D_{x_n}\op(\rho_1\chi_1)u-\op(\rho_2\chi_1)D_{x_n}u\nonumber\\
&\quad +\op(\rho_2\chi_1)\op(\rho_1\chi_1)u\nonumber\\
&=D_{x_n}^2u+\op(2i(\partial_{x_n}\VP)\chi_1 )D_{x_n}u+\op((q_2+2iq_1)\chi_1^2)u\nonumber\\
&\quad -[D_{x_n},\op(\rho_1\chi_1)]u+\op(R_1)u\label{eq 7, mu grand}
\end{align}
where $R_1\in hS^1_\T$ is given by symbolic calculus. 
\newline \noindent By Lemma \ref{regles}, the symbol of $[D_{x_n},\op(\rho_1\chi_1)]$ belongs to $hS^1_\T$. Thus, we have 
\begin{equation}\label{eq 8, mu grand}
(D_{x_n}-\op(\rho_2\chi_1))(D_{x_n}-\op(\rho_1\chi_1))u=f_2
\end{equation}                                
where 
\begin{align}
f_2= &P_\VP u-\op(2i(\partial_{x_n}\VP)(1-\chi_1))D_{x_n}\op(\chi_+)v-\op((q_2+2iq_1)(1-\chi_1^2))\op(\chi_+)v& \label{eq 8', mu grand}\\
& -[D_{x_n},\op(\rho_1\chi_1)]u+\op(R_1)u.&\nonumber
\end{align}
By  \eqref{eq 6, mu grand}, \eqref{eq 7, mu grand} and \eqref{eq 8', mu grand},  using that $(1-\chi_1)\chi_+=0$, $(1-\chi_1^2)\chi_+=0$ and $\|u\|_{H^1 }\le C\|v\|_{H^1 }$,
we obtain by symbolic calculus
\begin{align}
 \|f_2\|_{L^2(x_n>0)} \le C\|P_\VP v\|_{L^2(x_n>0)} +Ch\|v\|_{H^1(x_n>0)}.
\label{eq 9, mu grand}
\end{align}

\subsubsection{Estimate of $(D_{x_n}-\op(\rho_1\chi_1))u$}
Denoting $z=(D_{x_n}-\op(\rho_1\chi_1))u\in L^2(x_n>0)$ (since $g\in H^1(\R^n)$), we have by \eqref{eq 8, mu grand}
\begin{align}
 (D_{x_n}-\op(\rho_2\chi_1))\underline{z}=\underline{f_2}-ihz_{|x_n=0}\delta_{x_n=0}.\label{eq 11, mu grand}
\end{align}
Let $\chi\in S^0$ such that $\chi=1$ if $|\xi|$ is large, $\chi=1$ in a \nhd of $\supp \chi_+\times \R_{\xi_n}$ and $$\supp \chi\cap \{(x,\xi)\in \R^n\times\R^n,\ \xi_n-\rho_2(x,\xi')=0\}=\emptyset.$$
This is indeed possible. If $|\xi'|$ is large enough then $q_2(x,\xi')\ge C|\xi'|^2$ and $\Im \rho_2<-C|\xi'|$, $\mu(x,\xi')\ge 0$ and in this region $\xi_n-\rho_2(x,\xi')\not= 0$.
If $|\xi'|$ is bounded, $\rho_2(x,\xi')$ is also bounded and if $|\xi_n|$ large, $\xi_n-\rho_2(x,\xi')\not= 0$. Now, if $|\xi|$ is bounded then, on the support of $\chi_+$,
$\Im \rho_2<-\partial_{x_n}\VP(x)$ by Lemma~\ref{Lemme sur les racines} and $\xi_n-\rho_2(x,\xi')\not= 0$. 
\newline Moreover, by the same arguments, we obtain that $|\xi_n-\rho_2\chi_1|\ge c \est{\xi}$ on $\supp\chi$.

Observe now that $\xi_n-\rho_2\chi_1\in S\left(\est{\xi},dx^2+\frac{d\xi'^2}{\est{\xi'}^2}+\frac{d\xi_n^2}{\est{\xi}^2}\right)$.
The metric $\tilde g=dx^2+\frac{d\xi'^2}{\est{\xi'}^2}+\frac{d\xi_n^2}{\est{\xi}^2}$ is slowly varying, semi-classical $\sigma$-temperate,
the weights $\est{\xi}$, $\est{\xi'}$ are $\tilde g$-continuous and semi-classical $\sigma, \tilde g$-temperate (see definitions in Appendix A and Lemma~\ref{lem: bonne metrique bon poids} with $\eps=1$ with a change of variables and dimension).
\newline \noindent We set $$q(x,\xi) = \frac{\chi(x,\xi)}{\xi_n-(\rho_2\chi_1)(x,\xi')}\in S\left(\est{\xi}^{-1},dx^2+\frac{d\xi'^2}{\est{\xi'}^2}+\frac{d\xi_n^2}{\est{\xi}^2}\right).$$
We have, by symbolic calculus, 
$$\Op(q)\Op(\xi_n-\rho_2\chi_1)= \chi + R$$ where $R\in h  S\left(\est{\xi'}^{-1},dx^2+\frac{d\xi'^2}{\est{\xi'}^2}+\frac{d\xi_n^2}{\est{\xi}^2}\right)$. 
\newline \noindent Moreover, we can improve this by $R\in h S\left(\est{\xi}^{-1}, dx^2+\frac{d\xi'^2}{\est{\xi'}^2}+\frac{d\xi_n^2}{\est{\xi}^2}\right)$.
Indeed, the symbol of $\Op(q)\Op(\xi_n)$ is $q\xi_n$ and, using that $\rho_2\chi_1\in S\left(\est{\xi'},dx^2+\frac{d\xi'^2}{\est{\xi'}^2}+\frac{d\xi_n^2}{\est{\xi}^2}\right)$,
the symbol of $\Op(q)\Op(-\rho_2\chi_1)$ is $-q\rho_2\chi_1+R$ where  $R\in  h  S\left(\est{\xi}^{-1},dx^2+\frac{d\xi'^2}{\est{\xi'}^2}+\frac{d\xi_n^2}{\est{\xi}^2}\right)$.

\noindent Applying $\Op(q)$ in Formula \eqref{eq 11, mu grand}, we obtain

\begin{align}
 \underline{z}=\Op(q)\underline{f_2}-ih\Op(q)(z_{|x_n=0}\delta_{x_n=0})
+\Op(1-\chi)\underline{z}-\Op(R) \underline{z}.
\label{eq 12, mu grand}
\end{align}

In the sequel, we estimate each terms in the previous equality.
\newline \noindent First, we have
\begin{align}
&\|\Op(R)\underline{z}\|_{H^1(x_n>0)}\le \|\Op(R)\underline{z}\|_{1}\le Ch \|\underline{z}\|
 \le C h\|z\|_{L^2(x_n>0)},\nonumber\\
&\| \Op(q)\underline{f_2}\|_{H^1(x_n>0)}\le C\|\Op(q)\underline{f_2}\|_{1}\le C \| {f_2}\|_{L^2(x_n>0)}.\label{eq 13, mu grand}
\end{align}

\noindent Moreover, using Lemma \ref{regles}, we have
\begin{align}
 z=[D_{x_n}-\op(\rho_1\chi_1)]\op(\chi_+)v=\op(\chi_+)[D_{x_n}-\op(\rho_1\chi_1)]v+h\op(R_0)v
\label{eq 14, mu grand}
\end{align}
where $R_0\in S^0_\T$.

\noindent Let $y=[D_{x_n}-\op(\rho_1\chi_1)]v\in L^2(x_n>0)$. 
\newline \noindent We have, by \eqref{eq 14, mu grand}, $\underline{z}=\op(\chi_+)\underline{y}+h\underline{\op(R_0)v}$. Thus, we obtain
\begin{align}
 \Op(1-\chi)\underline{z}=\Op(1-\chi)\op(\chi_+)\underline{y}+h\Op(1-\chi)\underline{\op(R_0)v}.
\label{eq 15, mu grand}
\end{align}
Moreover, since $\supp(1-\chi)\cap \supp(\chi_+)=\emptyset$, one can apply \cite[Lemma 2.2]{LRR1} and get that 
$$\Op(1-\chi) \op(\chi_+) \in \bigcap_{N\in\mathbb{N}} h^N \Psi^{-N}$$
and, consequently,
\begin{align}
\|\Op(1-\chi)\op(\chi_+)\underline{y}\|_{1}\le h\|y\|_{L^2(x_n>0)}\le Ch\|v\|_{H^1(x_n>0)}.
\label{eq 16, mu grand}
\end{align}
We remark that $ D_{x_n}\Op(1-\chi) \in S^0$ because $\chi=1$ when $|\xi|$ large enough. Since $R_0$ is a tangential symbol, we get
\begin{align}
 \| D_{x_n}\Op(1-\chi)\underline{\op(R_0)v}\|_{L^2(x_n>0)}&\le \| D_{x_n}\Op(1-\chi)\underline{\op(R_0)v}\|\nonumber\\
&\le C \|\underline{\op(R_0)v}\|\le C\|v\|_{L^2(x_n>0)}.
\label{eq 16.1, mu grand}
\end{align}
 Following \eqref{eq 15, mu grand}, \eqref{eq 16, mu grand} and \eqref{eq 16.1, mu grand}, we have
\begin{align}
\|\Op(1-\chi)\underline{z}\|_{H^1(x_n>0)} \le \|\Op(1-\chi)\underline{z}\|_{1}\le Ch \|v\|_{H^1(x_n>0)}.
\label{eq 17, mu grand}
\end{align}

\noindent We have also
\begin{align}
 -ih\Op(q)(z_{|x_n=0}\delta_{x_n=0})&=\frac{1}{(2\pi h)^{n-1}}\int e^{ix'\xi'/h}\left(
\frac{1}{2i\pi}\int e^{ix_n\xi_n/h}q(x,\xi',\xi_n)d\xi_n,
 \right) \hat z_{|x_n=0}(\xi')d\xi'\nonumber\\
&=\op(t)(z_{|x_n=0})
\label{eq 19, mu grand}
\end{align}
where $\hat z_{|x_n=0}(\xi')$ is the Fourier transform with respect to $x'$ taken at $x_n=0$, the formula should be understood as an oscillating integral
and we have set
$$
 t(x,\xi')=\frac{1}{2i\pi}\int e^{ix_n\xi_n/h}q(x,\xi',\xi_n)d\xi_n.
$$
If one also requires $\chi=1$ for $|\xi_n|\ge C\est{\xi'}$ and  $\xi_n\in\C$ (which is compatible with the definition of $\chi$ on $\R$), we get
\begin{align}
 t(x,\xi')=\frac{1}{2i\pi}\int_\gamma e^{ix_n\xi_n/h}\frac{\chi(x,\xi',\xi_n)}{\xi_n-(\rho_2\chi_1)(x,\xi')}d\xi_n
\label{eq 20, mu grand}
\end{align}
where we integrate on the new contour $\gamma = [- C \est{\xi'}, C \est{\xi'}] \cup \{\xi_n \in \mathbb{C};
| \xi_n | = C \est{\xi'} , \Im(\xi_n)>0\}$.
\newline \noindent By \eqref{eq 20, mu grand}, we obtain that, for all $l\in\N$, all $\alpha,\ \beta\in \N^{n-1}$, there exists $C>0$ such that
\begin{align}
 |\partial_{x_n}^l\partial_{x'}^\alpha\partial_{\xi'}^\beta t(x,\xi')|\le Ch^{-l}\est{\xi'}^{l-|\beta|}.
\label{eq 21, mu grand}
\end{align}
Let now $\chi_2(x',\xi')\in\Cinf(\R^{n-1}\times \R^{n-1})$ a cut-off function such that $(\chi_+)_{|x_n=0}$ is supported in the interior of $\{ \chi_2=1\}$ and $(\chi_1)_{|x_n=0}=1$ on a \nhd of the support of $\chi_2$.
\newline \noindent One may write
\begin{align}
 \op(t)(z_{|x_n=0})=\op(t\chi_2)(z_{|x_n=0})+\op((1-\chi_2)t)(z_{|x_n=0}).
\label{eq 22, mu grand}
\end{align}
First, we get
\begin{align}
\op((1-\chi_2)t)(z_{|x_n=0})&=\op((1-\chi_2)t)\left[\left(D_{x_n}-\op(\rho_1\chi_1)\op(\chi_+)\right)v
\right]_{|x_n=0}\nonumber\\
&=\op((1-\chi_2)t)\op((\chi_+)_{|x_n=0})          \left( (D_{x_n}-\op(\rho_1\chi_1))v \right)_{|x_n=0} \nonumber\\
&\quad+\op((1-\chi_2)t)\left( \left[ 
D_{x_n}-\op(\rho_1\chi_1),\op(\chi_+)
\right] v
\right) _{|x_n=0}.
\label{eq 22.1, mu grand}
\end{align}
By symbolic calculus and as $\supp (1-\chi_2)\cap \supp\chi_+=\emptyset$, the asymptotic expansion of the symbols of $\op((1-\chi_2)t)\op((\chi_+)_{|x_n=0}) $ and $\op((1-\chi_2)t)\left[ 
D_{x_n}-\op(\rho_1\chi_1),\op(\chi_+)\right]_{|x_n=0}$ are null (taking account that $\left[ 
D_{x_n}-\op(\rho_1\chi_1),\op(\chi_+)\right]$ is a tangential operator). Hence by trace formula, we have
\begin{align}
\|\op(\est{\xi'}) \op((1-\chi_2)t)(z_{|x_n=0})\|_{L^2(x_n>0)}\le Ch\|v\|_{H^1(x_n>0)}+Ch|D_{x_n}v_{|x_n=0}|_{-1/2}.
\label{eq 23, mu grand}
\end{align}

\noindent On the support of $\chi_2$ we have $\chi_{|x_n=0}=0$ and, following \eqref{eq 20, mu grand}, we deduce
\begin{align}
(t\chi_2)(x,\xi')= \frac{1}{2i\pi}\chi_2(x,\xi')\int_\gamma e^{ix_n\xi_n/h}\frac{1}{\xi_n-\rho_2(x,\xi')}d\xi_n=0
\label{eq 24, mu grand}
\end{align}
by residue formula and since, by Lemma \ref{Lemme sur les racines}, $\Im \rho_2<0$ on the support of $\chi_2$.

To estimate the $L^2$ norm of $\d_{x_n}\op(t)(z_{|x_n=0})=\op(\d_{x_n}t)(z_{|x_n=0})$, we proceed in the same way. 
Actually, $\d_{x_n}t\in h^{-1}S_\T^1$ and we have to use \eqref{eq 22.1, mu grand}.
\newline \noindent By the same support argument used to obtain \eqref{eq 23, mu grand}, we get
\begin{align}
\|\d_{x_n} \op((1-\chi_2)t)(z_{|x_n=0})\|_{L^2(x_n>0)}\le Ch\|v\|_{H^1(x_n>0)}+Ch|D_{x_n}v_{|x_n=0}|_{-1/2}.
\label{eq 24.1, mu grand}
\end{align}
Analogously, the equation \eqref{eq 24, mu grand} become
\begin{align}
(\d_{x_n}t\chi_2)(x,\xi')=\frac{ h^{-1}}{2\pi}\chi_2(x,\xi')\int_\gamma e^{ix_n\xi_n/h}\frac{\xi_n}{\xi_n-\rho_2(x,\xi')}d\xi_n=0.
\label{eq 26.1, mu grand}
\end{align}
Following \eqref{eq 22, mu grand},  \eqref{eq 23, mu grand},  \eqref{eq 24, mu grand}, \eqref{eq 24.1, mu grand} and \eqref{eq 26.1, mu grand}, we deduce
\begin{align}
\| \op(t)(z_{|x_n=0})\|_{H^1(x_n>0)}\le Ch\|v\|_{H^1(x_n>0)}+Ch|D_{x_n}v_{|x_n=0}|_{-1/2}.
  \label{eq 25, mu grand}
\end{align}

\noindent Finally, using \eqref{eq 12, mu grand}, \eqref{eq 13, mu grand},  \eqref{eq 17, mu grand},  
\eqref{eq 19, mu grand}, 
\eqref{eq 25, mu grand}, and for all $h$ small enough, we obtain
\begin{eqnarray}\label{eq 26, mu grand}
 \|\ z\|_{H^1(x_n>0)}&\le& C\|f_2\|_{L^2(x_n>0)}+Ch\|v\|_{H^1(x_n>0)}+Ch
|(D_{x_n}v)_{|x_n=0}|_{-1/2}\nonumber \\  
&\le& C\|P_\VP v\|_{L^2(x_n>0)}+Ch\|v\|_{H^1(x_n>0)}+Ch
|(D_{x_n}v)_{|x_n=0}|_{-1/2}
\end{eqnarray}
where we have used \eqref{eq 9, mu grand}.

\subsubsection{Estimates of $v_0$ and $v_1$}
The goal is now to find an equation on $v_0$ and $v_1$ (see their definitions in \eqref{eq 2 : mu grand}) . We remind that $z=[D_{x_n}-\op(\rho_1\chi_1)]u $ and, since $u=\op(\chi_+)v$, we have then
\begin{align}
 &(D_{x_n}v)_{|x_n=0}-(\op(\rho_1\chi_1)v)_{|x_n=0}=f_3\nonumber\\
&\textrm{where } f_3=z_{|x_n=0}+(1-\op(\chi_+))(D_{x_n}v)_{|x_n=0}-\op(\rho_1\chi_1)(1-\op(\chi_+))v_{|x_n=0}.
\label{eq 27, mu grand}
\end{align}
Following \eqref{eq 3 : mu grand}, we have 
\begin{align}
 &v_1-\op((\rho_1+i(\partial_{x_n}\VP)_{|x_n=0})\chi_1)v_0=f_4\nonumber\\
&\textrm{where }f_4=f_3-G_1+\op((\rho_1\chi_1)_{|x_n=0})G_0+\op(i(\partial_{x_n}\VP)_{|x_n=0}(1-\chi_{1 |x_n=0}))v_0.
\label{eq 28, mu grand}
\end{align}
\begin{remark}
We may now explain the main difficulty faced to solve this equation. Coming back to our model case detailled in the beginning of section 2, one has
$$\rho_1+i (\partial_{x_n}\VP)_{|x_n=0}= i |\xi'|,$$
so that equation \eqref{eq 28, mu grand} takes the form, up to some remainder term and where $f$ is some data,
$$v_1-i \op(|\xi'|)v_0=f.$$
Since the symbol $|\xi'|$ does not satisfy the transmission condition, one cannot use the usual algebra of pseudodifferential operators.
We will overcome this problem writing a factorization of the form
$$|\xi'|=(\xi_1+i |\xi''|)^{1/2} (\xi_1-i |\xi''|)^{1/2} $$
and using that the operators of symbols $(\xi_1\pm i |\xi''|)^{1/2}$ preserves functions with support in $\{\mp x_1>0\}$.
\end{remark}
Coming back to our remainder estimates, one has
$$
 |D_{x_n}v_{|x_n=0}|_{-1/2}\le C|v_0|_{1/2}+C|v_1|_{-1/2}+C|G_1|_{-1/2}
$$
and, since $1-\chi_{1 |x_n=0}=(1-\chi_{1 |x_n=0})( 1-\chi_+)$,
\begin{equation}\label{difference}
 \op(1-\chi_{1 |x_n=0})-\op((1-\chi_{1 |x_n=0}))\op( 1-\chi_+)\in h \Psi^0_\T
\end{equation}
which gives
$$
 |\op{(1-\chi_{1 |x_n=0}})v_0|\le C|\op{(1-\chi_+)v_0}|+Ch|v_0|.
$$
Consequently, by \eqref{eq 26, mu grand}, 
\eqref{eq 27, mu grand},
 \eqref{eq 28, mu grand} 
and trace formula, we obtain

\begin{align}
 h^{1/2}|f_4|_{-1/2}&\le Ch^{1/2}|G_1|_{-1/2}+Ch^{1/2}|G_0|_{1/2}+C\|P_\VP v\|_{L^2(x_n>0)}+Ch\|v\|_{H^1(x_n>0)}+Ch|v_1|_{-1/2}
\nonumber\\
&\quad+Ch|v_0| +Ch^{1/2}|\op{(1-\chi_+)}D_{x_n}v_{|x_n=0}|_{-1/2}+Ch^{1/2}|\op{(1-\chi_+)}v_{|x_n=0}|_{1/2}.
\label{eq 29, mu grand}
\end{align}

\noindent Let now
\begin{align}
 &\lambda_-^s(\xi')=\left( \xi_1+i\sqrt{|\xi''|^2+(\eps \partial_{x_n}\VP(0))^2} \right)^s\nonumber,\\
&\lambda_+^s(\xi')=\left( \xi_1-i\sqrt{|\xi''|^2+(\eps \partial_{x_n}\VP(0))^2} \right)^s\nonumber.
\end{align}
Since $\lambda_-^s$ is holomorphic function in $\Im \xi_1>0$ and using an adapted version of the Paley-Wiener theorem (see Theorem 7.4.3 in \cite{Hoermander:V3}),
one gets that $\lambda_-^s$ is the Fourier transform of a distribution supported on $\{x_1\le0\}$ and, for analogous reason, 
$\lambda_+^s$ is the Fourier transform of a distribution supported on $\{x_1\ge0\}$. This justifies the indices $+$ and $-$.
\par We set $z_1=\op(\lambda_-^{-1/2})v_1\in L^2(x_n=0)$ and  $z_0=\op(\lambda_+^{1/2})v_0\in L^2(x_n=0)$. 
We have $\supp z_1\subset\{x_1\le0\}$ and $\supp z_0\subset\{x_1\ge0\}$. Moreover, by  \eqref{eq 28, mu grand},
$$
z_1-\op(\lambda_-^{-1/2})\op(((\rho_1+i\partial_{x_n}\VP)\chi_1)_{|x_n=0})\op(\lambda_+^{-1/2})z_0=f_5
$$
where 
\begin{equation}\label{eq 33, mu grand}h^{1/2}|f_5|\le C_\eps h^{1/2} |f_4|_{-1/2}.\end{equation}
As $z_1$ supported in $\{x_1\le0\}$, we have
\begin{align}
 r_{x_1>0}\op(\lambda_-^{-1/2})\op(((\rho_1+i\partial_{x_n}\VP)\chi_1)_{|x_n=0})\op(\lambda_+^{-1/2})z_0=-r_{x_1>0}f_5
\label{eq 34, mu grand}
\end{align}
where we denote here $r_{x_1>0}z:=z_{|x_1>0}$ the restriction of $z$ to $\{x_1>0\}$.

\noindent Following the notations introduced in Appendix \ref{Appendix A}, we note that $\lambda_\pm^{-1/2}\in S(\este{\xi}^{-1/2}, g)$.
\newline Moreover, one has $(\rho_1+i(\partial_{x_n}\VP)_{|x_n=0})\chi_1\in S(\este{\xi}, g)$. Consequently, by symbolic calculus, we have
\begin{equation}\label{eq 35, mu grand}
\op(\lambda_-^{-1/2})\op(((\rho_1+i\partial_{x_n}\VP)\chi_1)_{|x_n=0})\op(\lambda_+^{-1/2})
=\op(a)+h\op(b)
\end{equation}
where 
$$\left \{ \begin{array}{l} a=\lambda_-^{-1/2}((\rho_1+i\partial_{x_n}\VP)\chi_1)_{|x_n=0}\lambda_+^{-1/2},\\
   b\in S(\est{\xi'}\este{\xi'}^{-1}\este{\xi''}^{-1},g).
  \end{array}
\right.
$$

\begin{remark} To guide the reader, we shall also explain what happens at this milestone in our model case. As explained above, one has
$$(\rho_1+i\partial_{x_n}\VP)=i |\xi'|$$ 
and consequently, in some sense precised below,
$$a=i+O(\eps).$$
The very simple form of this operator explain that one should now get the estimates desired on $z_0$ and then on $v_0$ and $v_1$.
\end{remark}

\noindent The estimates on $b$ implies that $|b(x',\xi')|\le C/\eps^2$ (because 
$\est{\xi'}\este{\xi'}^{-1}\este{\xi''}^{-1}\le \eps^{-2}$) and $b\in S(1,g) $ 
with semi-norms depending on $\eps$. We can apply the Lemma \ref{lemme estimation norme} to get
\begin{align}\forall \eps>0,\ \exists C_\eps>0;\ |\op(b)z_0|_{L^2(x_1>0)}\le |\op(b)z_0|\le C_\eps |z_0|.
\label{eq 37, mu grand}
\end{align}

\noindent Second, we have $a=a_{|x'=0}+c$ and $a\in S(\est{\xi'}\este{\xi'}^{-1},g)$ which imply that $a\in S(1,g) $ with semi-norms depending on $\eps$ and $|\partial_{x'}c(x',\xi')|\le C\eps^{-1}$
 (because $\est{\xi'}\este{\xi'}^{-1}\le \eps^{-1}$).\\
We can apply Lemma \ref{estimation symbole nul en 0}
to $c$ and obtain, if the radius satisfies $\kappa\le \eps^2$ (see the hypotheses of Theorem 2) and since $v_0$ is supported in $B_\kappa$, 
\begin{align}
 |\op(c)z_0|_{L^2(x_1>0)}\le C\eps|z_0|+C_\eps h^{1/2}|v_0|_{1/2}.
\label{eq 38, mu grand}
\end{align}
Next, we use the assumption $|\partial_{x'}\VP(0)|\le \eps \partial_{x_n}\VP(0)$, the change of variables found in Lemma \ref{Lemme forme normale} and  Lemma \ref{Lemme sur les racines}. 
\newline \noindent Writing $V=\frac{\partial_{x'}\VP(0)}{\partial_{x_n}\VP(0)}$ and $\eta'=\frac{\xi'}{\partial_{x_n}\VP(0)}$, we have by \eqref{racines x=0,1}
$$\rho_1(0,\xi')+i\partial_{x_n}\VP(0)=i \partial_{x_n}\VP(0) \rho(\eta')$$
with
$$ \rho(\eta')=\sqrt{|\eta'|^2-|V|^2+2i\eta'.V} \ .$$
One may write
\begin{equation}\label{eq 39, mu grand}
a_{|x'=0}(\xi')=i\frac{ \partial_{x_n}\VP(0) \ \rho(\eta') \ (\chi_1)_{|x=0}}{\sqrt{|\xi'|^2+(\eps\partial_{x_n}\VP(0))^2}}
=i(\chi_1)(0,\xi')+id(\xi')
\end{equation}
where $d(\xi')=\tilde{d}(\eta')\times \chi_1(0,\xi')$ and
 $$\tilde{d}(\eta')=\frac{\rho(\eta')}{\sqrt{|\eta'|^2+\eps^2}}-1=
\frac{\rho(\eta')-\sqrt{|\eta'|^2+\eps^2}}{\sqrt{|\eta'|^2+\eps^2}}.$$
We remark that
$$\forall \eta', \ |\eta'|\le \sqrt{|\eta'|^2+\eps^2}\le |\eta'|+\eps$$
and, using now \eqref{racines x=0,2}, we thus obtain, since $|V|\le \eps$ and for $\eps$ sufficiently small,
$$ |\tilde{d}(\eta')|\le C \eps$$
since, by \eqref{racines x=0,3}, $|\eta'|\ge \delta$ on $\supp(\chi_1)\subset \mathcal{F}_{\alpha}$. 
\newline \noindent Finally, one has
$$|d(\xi')|\le C\eps \text{  and } d(\xi')\in S(1,g)$$
and we can apply Lemma \ref{lemme estimation norme} to deduce
\begin{align}
 |\op(d)z_0|_{L^2(x_1>0)}\le |\op(d)z|_0\le  (C\eps+C_\eps h^{1/2}) |z_0|.
\label{eq 40, mu grand}
\end{align}

\noindent Following \eqref{eq 34, mu grand}, \eqref{eq 35, mu grand} and \eqref{eq 39, mu grand}, we obtain
\begin{align}
 ir_{x_1>0}z_0+r_{x_1>0}\op(hb+c+id (\chi_1)_{|x_n=0}+(1-(\chi_1)_{|x_n=0}))z_0=-r_{x_1>0}f_5.
\label{eq 41, mu grand}
\end{align}
On the other hand, using \eqref{difference} again, we have
\begin{eqnarray*}
|\op(1-(\chi_1)_{|x_n=0})z_0|&\le& C|\op(1-\chi_+)z_0|+Ch |z_0|\\
& \le & C|\op(1-\chi_+)v_0|_{1/2}+Ch |v_0|_{1/2}
\end{eqnarray*}
and, following \eqref{eq 37, mu grand}, \eqref{eq 38, mu grand}, \eqref{eq 40, mu grand} and \eqref{eq 41, mu grand}, we deduce
\begin{eqnarray*}
 |z_0|_{L^2(x_1>0)}&\le& (C\eps+C_\eps h^{1/2})|z_0|+C_\eps h^{1/2}|v_0|_{1/2}+C|\op(1-\chi_+)v_0|_{1/2}+|f_5|\\
&\le& C\eps|z_0|+C_\eps h^{1/2}|v_0|_{1/2}+C|\op(1-\chi_+)v_0|_{1/2}+|f_5|.
\end{eqnarray*}
\noindent It is clear that $ |z_0|_{L^2(x_1>0)}= |z_0|$. Taking $\eps$ small enough, we have then
$$
\forall h\in (0,h_0], \ |z_0|\le C h^{1/2}|v_0|_{1/2}+C|\op(1-\chi_+)v_0|_{1/2}+C|f_5|.
$$
Using $z_0=\op(\lambda_+^{1/2})v_0$ and for $h_0$ small enough, we deduce 
$$
 |v_0|_{1/2}\le C|\op(1-\chi_+)v_0|_{1/2}+C |f_5|.
$$
Following \eqref{eq 29, mu grand} and \eqref{eq 33, mu grand} we have, using \eqref{eq 3 : mu grand} and for $h_0 $ small enough,
\begin{align}
 h^{1/2} |v_0|_{1/2}&\le Ch^{1/2}|G_1|_{-1/2}+Ch^{1/2}|G_0|_{1/2}+C\|P_\VP v\|_{L^2(x_n>0)}+Ch\|v\|_{H^1(x_n>0)}+Ch|v_1|_{-1/2}
\nonumber\\
&\quad +Ch^{1/2}|\op{(1-\chi_+)}D_{x_n}v_{|x_n=0}|_{-1/2}+Ch^{1/2}|\op{(1-\chi_+)}v_{|x_n=0}|_{1/2}.
\label{eq 45, mu grand}
\end{align}

\noindent Using \eqref{eq 45, mu grand} in \eqref{eq 28, mu grand} and by \eqref{eq 29, mu grand}, we obtain also, for $h_0$ small enough,

\begin{align*}
h^{1/2} |v_1|_{-1/2}&\le Ch^{1/2}|G_1|_{-1/2}+Ch^{1/2}|G_0|_{1/2}+C\|P_\VP v\|_{L^2(x_n>0)}+Ch\|v\|_{H^1(x_n>0)}+Ch|v_1|_{-1/2}
\nonumber\\
&\quad +Ch^{1/2}|\op{(1-\chi_+)}D_{x_n}v_{|x_n=0}|_{-1/2}+Ch^{1/2}|\op{(1-\chi_+)}v_{|x_n=0}|_{1/2}.
\end{align*}
One now remarks that the term $|v_1|$ on the right-hand side of this inequality can be absorbed if $h_0$ is sufficiently small.
Recalling now \eqref{eq 3 : mu grand} and \eqref{eq 4 : mu grand}, one gets, provided $h_0$ is small enough,
\begin{align}
h^{1/2} |v_{|x_n=0}|_{1/2}&\le Ch^{3/2}|e^{\VP/h}g_1|_{-1/2}+Ch^{1/2}|e^{\VP/h}g_0|_{1/2}+C\|P_\VP v\|_{L^2(x_n>0)}+Ch\|v\|_{H^1(x_n>0)}
\nonumber\\
&\quad +Ch^{1/2}|\op{(1-\chi_+)}D_{x_n}v_{|x_n=0}|_{-1/2}+Ch^{1/2}|\op{(1-\chi_+)}v_{|x_n=0}|_{1/2}
\label{eq 46, mu grand}
\end{align}
and
\begin{align}
h^{1/2} |D_{x_n}v_{|x_n=0}|_{-1/2}&\le Ch^{3/2}|e^{\VP/h}g_1|_{-1/2}+Ch^{1/2}|e^{\VP/h}g_0|_{1/2}+C\|P_\VP v\|_{L^2(x_n>0)}+Ch\|v\|_{H^1(x_n>0)}
\nonumber\\
&\quad +Ch^{1/2}|\op{(1-\chi_+)}D_{x_n}v_{|x_n=0}|_{-1/2}+Ch^{1/2}|\op{(1-\chi_+)}v_{|x_n=0}|_{1/2}.
\label{eq 47, mu grand}
\end{align}

\subsubsection{Estimate of $u$ in $H^1(x_n>0)$}
To estimate the $H^1$-norm of $u$ we use the Carleman technique.
\\First, we have
\begin{eqnarray}
 \|(D_{x_n}-\op(\rho_1\chi_1))u\|_{L^2(x_n>0)}^2&=&
\|(D_{x_n}-\op(\Re\rho_1\chi_1))u\|_{L^2(x_n>0)}^2+
\|\op(\Im\rho_1\chi_1)u\|_{L^2(x_n>0)}^2\nonumber\\
&+&2\Re((D_{x_n}-\op(\Re\rho_1\chi_1))u|-i\op(\Im\rho_1\chi_1)u)_{L^2(x_n>0)}\nonumber\\
&=&\|(D_{x_n}-\op(\Re\rho_1\chi_1))u\|_{L^2(x_n>0)}^2
+\|\op(\Im\rho_1\chi_1)u\|_{L^2(x_n>0)}^2\nonumber\\
& +&\Re(i[\op(\Im\rho_1\chi_1),D_{x_n}-\op(\Re\rho_1\chi_1)]u|u) \nonumber   \\
&-&h\Re(\op(\Im\rho_1\chi_1)u|u)_0\nonumber\\
&+&h \left(\big(L_0(D_{x_n}-\op(\Re\rho_1\chi_1))+L_1\op(\Im\rho_1\chi_1)\big)u|u\right),
\label{Carleman sur u}
\end{eqnarray}
where $(.,.)$ (resp. $(.,.)_0$) denotes the standard scalar product on $\{x_n>0\}$ (resp. $\{x_n=0\}$) and $L_j$ are tangential operators of orders 0.
\\Moreover, one has
\begin{equation}\label{Terme erreur Carleman} | (\op(\Im\rho_1\chi_1)u|u)_0 |\le C|u_{|x_n=0}|^2_{1/2} \end{equation}
and, for any $C'>0$,
\begin{align}
&h\left| \big((L_0(D_{x_n}-\op(\Re\rho_1\chi_1))+L_1\op(\Im\rho_1\chi_1)\big)u|u)\right|\nonumber \\
&\quad \quad\  \ \le \frac{1}{C'}\big(  
\|(D_{x_n}-\op(\Re\rho_1\chi_1))u\|_{L^2(x_n>0)}^2+
\|\op(\Im\rho_1\chi_1)u\|_{L^2(x_n>0)}^2
\big) +Ch^2\|u \|_{L^2(x_n>0)}^2.\label{Terme 2 erreur Carleman}
\end{align}
The commutator term $ i[\op(\Im\rho_1\chi_1),D_{x_n}-\op(\Re\rho_1\chi_1)]$ is a tangential 1-order operator and has $h\{\Im\rho_1\chi_1,\xi_n- \Re\rho_1\chi_1\}$ for principal symbol. 
\\Taking account that $p_\varphi=(\xi_n-\rho_1)(\xi_n-\rho_2)$, we have
\begin{align*}
 \{ \Re p_\varphi, \Im p_\varphi \}&=\frac1{2i}\{\bar p_\varphi ,p_\varphi\}\\
&=|\xi_n-\rho_2|^2\{\Im \rho_1,\xi_n-\Re (\rho_1)\}+O(\xi_n-\rho_1)+O(\xi_n-\bar \rho_1).
\end{align*}
Hence by hypothesis \eqref{hypohormander} we obtain that $$\Im \rho_1=0, \ \xi_n=\Re \rho_1 \ \Rightarrow \ \{\Im \rho_1,\xi_n-\Re (\rho_1)\}>0.$$ 
Furthermore, noting that $\{\Im \rho_1,\xi_n-\Re (\rho_1)\}$ is in fact independent of $\xi_n$, we have
$$\Im \rho_1=0 \ \Rightarrow \ \{\Im \rho_1,\xi_n-\Re (\rho_1)\}>0$$ 
and this classically implies that there exists $C_1>0$ and $C_2>0$ such that, on the support of $\chi_1$, 
\begin{align*}
 \{\Im \rho_1,\xi_n-\Re (\rho_1)\}+C_1\langle \xi'\rangle^{-1}|\Im \rho_1|^2\ge C_2\langle \xi'\rangle.
\end{align*}
Using the standard G\aa rding inequality (see e.g. \cite[Theorem 3.5.8]{Martinez:02}), we deduce
\begin{eqnarray}\label{Commutateur positif}&C_2 \| \op(\langle \xi'\rangle^{1/2})\op(\chi_1)u  \|_{L^2(x_n>0)}^2-Ch\|u\|_{L^2(x_n>0)}^2&\\
\le &\Re(i[\op(\Im\rho_1\chi_1),D_{x_n}-\op(\Re\rho_1\chi_1)]u|u)
+C_1\|\op(\Im \rho_1\chi_1)u\|_{L^2(x_n>0)} ^2.&   \nonumber 
\end{eqnarray}

\noindent On the other hand, as $\chi_1=1$ on the support of $\chi_+$, we have $\op(\chi_+)=\op(\chi_1)\op(\chi_+)+h\op(r_{-1})$ where $r_{-1}\in S^{-1}_\T$.
Thus
\begin{equation}\label{Rapport u et v}
 \|\op(\langle \xi'\rangle^{1/2})u  \|_{L^2(x_n>0)}^2\le C\|\op(\langle \xi'\rangle^{1/2}) \op(\chi_1)u  \|_{L^2(x_n>0)}^2+Ch\|v\|_{L^2(x_n>0)}^2.
\end{equation}
From \eqref{Carleman sur u}, \eqref{Terme erreur Carleman}, \eqref{Terme 2 erreur Carleman}, \eqref{Commutateur positif} and \eqref{Rapport u et v}, we get
\begin{eqnarray}\label{Carleman basses frequences}
&\frac12\|\op(\Im\rho_1\chi_1)u\|_{L^2(x_n>0)}^2+C_2 \| u  \|_{H^{1/2}(x_n>0)}^2 &\\
&\le \|(D_{x_n}-\op(\rho_1\chi_1))u\|_{L^2(x_n>0)}^2+Ch|u_{|x_n=0}|_{1/2}^2+
Ch\|v\|_{L^2(x_n>0)}^2\nonumber& 
\end{eqnarray}
for $C'$ chosen sufficiently large.
\\ \noindent Let now $C> 0$ given by Lemma~\ref{Lemme sur les racines} such that $\Im \rho_1(x,\xi')\ge \delta |\xi'|$ for $|\xi'|\ge C$. Let also $\chi_H\in \Con^\infty(\R^{n-1})$ such that 
$$\chi_H(\xi')=\begin{cases}
                1 \text{ if } |\xi'|\ge C+1,\\
0 \text{ if } |\xi'|\le C.
               \end{cases}
$$
Note that in particular $\chi_1=1$ on the support of $\chi_H$ so that
$$\op(\langle \xi'\rangle)\op(\chi_H)- \op(\chi_H \langle \xi'\rangle/(\Im \rho_1))\op((\Im \rho_1)\chi_1)\in h \Psi^0_\T.$$
Using symbolic calculus, we consequently obtain
\begin{align}\label{Hautes frequences}
 \| \op(\langle \xi'\rangle)\op(\chi_H)u\|_{L^2(x_n>0)}\le C\|\op((\Im \rho_1)\chi_1))u\|_{L^2(x_n>0)}+Ch\|u\|_{L^2(x_n>0)}.
\end{align}
If $h_0$ is small enough, using \eqref{Carleman basses frequences}, \eqref{Hautes frequences} and that $1-\chi_H$ is compactly supported, we get
\begin{align}\label{Carleman tangentiel}
 \| \op(\langle \xi'\rangle)u\|_{L^2(x_n>0)}&\le   \| \op(\langle \xi'\rangle)\op(\chi_H)u\|_{L^2(x_n>0)}+ \| \op(\langle \xi'\rangle)(1-\op(\chi_H))u\|_{L^2(x_n>0)}\nonumber\\
& \le C \|(D_{x_n}-\op(\rho_1\chi_1))u\|_{L^2(x_n>0)}+Ch^{1/2}|u_{|x_n=0}|_{1/2}+Ch^{1/2}\|v\|_{L^2(x_n>0)}.
\end{align}
We have also
\begin{align}\label{Estimation normale}
 \| D_{x_n}u\|_{L^2(x_n>0)}&\le  \| (D_{x_n}-\op(\rho_1\chi_1))u\|_{L^2(x_n>0)}+ \| \op(\rho_1\chi_1)u\|_{L^2(x_n>0)}\nonumber\\
&\le \| (D_{x_n}-\op(\rho_1\chi_1))u\|_{L^2(x_n>0)}+ C\|\op(\langle \xi'\rangle)u\|_{L^2(x_n>0)}.
\end{align}
Then, by \eqref{Carleman tangentiel} and \eqref{Estimation normale}, one gets
$$
 \| \op(\chi_+)v\|_{H^1(x_n>0)}\le C \|(D_{x_n}-\op(\rho_1\chi_1))u\|_{L^2(x_n>0)}+Ch^{1/2}|u_{|x_n=0}|_{1/2}+Ch^{1/2}\|v\|_{L^2(x_n>0)}.
$$
Using now \eqref{eq 26, mu grand}, we finally deduce
\begin{align}\label{estimation interieur +}
&\| \op(\chi_+)v\|_{H^1(x_n>0)}\nonumber \\
&\qquad \le C \left(\|P_\VP v\|_{L^2(x_n>0)}+h^{1/2}\|v\|_{H^1(x_n>0)}+h^{1/2}| v_{|x_n=0}|_{1/2}+h | (D_{x_n} v)_{|x_n=0}|_{-1/2}\right).
\end{align}
\subsection{End of the proof}
We now collect the results of Sections 2.1 and 2.2.\\
We first note that is possible to choose $\alpha$ so small that $\chi_-=1$ on $\supp(1-\chi_+)$.
\newline \noindent Doing so, one has $ \op{(1-\chi_+)}-\op{(1-\chi_+)}\op{(\chi_-)} \in h \Psi^0_\T$
which gives
\begin{eqnarray*}&|\op{(1-\chi_+)}D_{x_n}v_{|x_n=0}|_{-1/2}+|\op{(1-\chi_+)}v_{|x_n=0}|_{1/2}& \\
& \le C \left(|\op{(\chi_-)}D_{x_n}v_{|x_n=0}|_{-1/2}+|\op{(\chi_-)}v_{|x_n=0}|_{1/2} + h|D_{x_n}v_{|x_n=0}|_{-1/2}+h|v_{|x_n=0}|_{1/2}\right).&
\end{eqnarray*}
Summing up \eqref{traces elliptiques} and \eqref{eq 46, mu grand}, \eqref{eq 47, mu grand}, we now deduce the final trace estimate
\begin{eqnarray}\label{traces v}
&|D_{x_n}v_{|x_n=0}|_{-1/2}+|v_{|x_n=0}|_{1/2}&\nonumber\\
&\le C\left(h|e^{\VP/h}g_1|+|e^{\VP/h}g_0|_{1}+h^{-1/2}\|P_\VP v\|_{L^2(x_n>0)}+h^{1/2}\|v\|_{H^1(x_n>0)}\right)&
\end{eqnarray}
provided $h_0$ is sufficiently small.
\par Proceeding in the same way, one can deduce that, provided $h$ is sufficiently small,
\begin{eqnarray*}
\|v\|_{H^1(x_n>0)}\le C \left(\|\op(\chi_-) v\|_{H^1(x_n>0)}+\|\op(\chi_+)v\|_{H^1(x_n>0)}\right).
\end{eqnarray*}
Consequently, using \eqref{traces elliptiques}, \eqref{estimation interieur +} and \eqref{traces v}, one gets the desired result
$$
\|v\|_{H^1(x_n>0)}+ |D_{x_n}v_{|x_n=0}|_{-1/2}+|v_{|x_n=0}|_{1/2}\le C\left(|e^{\VP/h}g_0|_{1/2}+h|e^{\VP/h}g_1|_{-1/2}+h^{-1/2}\|P_\VP v\|_{L^2(x_n>0)}\right).
$$

%
%

\section{Proof of Theorem \ref{decroissancelog} }
\subsection{Different Carleman estimates}
We define $X = (- 1, 1) \times \Omega$ (where $(-1,1)$, diffeomorphic to $\R$, is considered as a manifold without boundary) and we split its boundary into $\partial X_N = (- 1, 1) \times \partial \Omega_N$
and $\partial X_D = (- 1, 1) \times \partial \Omega_D$.
\par Since $\Omega$ is a relatively compact smooth open set of $\R^d$, there exists some smooth function $f$ defined in a \nhd of $\partial \Omega$ such that, near $\partial \Omega$,
$$\left\{\begin{array}{l}
y \in \partial \Omega \Leftrightarrow f(y)=0, \\
y \in \Omega \Leftrightarrow f(y)>0. 
\end{array}
\right.$$
Moreover, near any point $y^*\in \partial \Omega$, one has $df\neq 0$. We may apply Lemma \ref{Lemme forme normale} with $p(y,\xi)=|\xi|^2$, transport $y^*$ to $0$ and get that, in new coordinates,
\begin{equation}\label{bord}
\left\{\begin{array}{l}
y \in \Omega \Leftrightarrow y_d>0, \\
y \in \partial \Omega \Leftrightarrow y_d=0,\\
\end{array}
\right. 
\end{equation}
with moreover $p(y,\xi)=\xi_d^2+r(y,\xi')$.
\newline \noindent The operator $-\Delta$ consequently takes, in some neighbourhood of $0$, the form $- \partial^2_{y_d} +l(y)\partial_{y_d}- Q \left( y, \partial_{y'} \right) $
with $\partial_{y'} = (\partial_{y_1}, \ldots, \partial_{y_{d - 1}})$ and $Q$ a smooth elliptic differential operator of order $2$.
\newline As in \cite{LR2}, one can find some function $e (y', y_d)$ normalised by $e (y', 0) = 1$ such
that the operator $P= - e \circ (\partial_{x_0}^2+ \Delta) \circ 1 / e$ takes, in some neighbourhood of $0$ and in
coordinates $x = (x_0, y)$, the form
\begin{equation}\label{formeoperateur}
  P = - \partial^2_{x_d} + R \left( x,  \frac{1}{i} \partial_{x'} \right)
\end{equation}
where $\partial_{x'} = (\partial_{x_0}, \ldots, \partial_{x_{d - 1}})$, the principal symbol of $R$ is real and satisfy, for some $c>0$,
\begin{eqnarray}\label{formeoperateur2}
 \forall (x,\xi'),\ \left\{\begin{array}{l}
    r (x, \xi') \geqslant c| \xi'|^2,\\
r(x,\xi')=|\xi_0|^2+r_0(x,\xi'_0), \end{array} \right.
\end{eqnarray}
with $\xi'_0 = (\xi_1, \ldots, \xi_{d - 1})$. 
\newline \noindent Moreover,  since $\Gamma$ is a relatively compact smooth submanifold of $\partial \Omega$, there exists some smooth function $k$ defined in a \nhd of $\Gamma$ such that, near $\Gamma$,
$$\left\{\begin{array}{l}
y \in \partial \Omega_D \Leftrightarrow f(y)=0 \text{ and } k(y)>0, \\
y \in \partial \Omega_N \Leftrightarrow f(y)=0 \text{ and } k(y)<0. 
\end{array}
\right .$$
Hence, if one works near $y^*\in \Gamma$, one has $dk(y^*)\neq 0$ and we additionally obtain that
\begin{equation}\label{borddubord}
\left\{\begin{array}{l}
y \in \partial \Omega_D \Leftrightarrow y_d=0 \text{ and } y_1>0,\\
y \in \partial \Omega_N \Leftrightarrow y_d=0 \text{ and } y_1<0,
\end{array}
\right. \ \text{ that is }  
\left\{\begin{array}{l}
x \in \partial X_D \Leftrightarrow x_d=0 \text{ and } x_1>0, \\
x \in \partial X_N \Leftrightarrow x_d=0 \text{ and } x_1<0, \\
\end{array}
\right. 
\end{equation}
and, along with \eqref{formeoperateur}, \eqref{formeoperateur2}, that $$r(0,\xi')=|\xi'|^2.$$
\par On the other hand, if $v$ satisfies \eqref{ondesabstrait}, then $\tilde{v}=e \times v$ satisfies the following equations
\begin{equation} \label{probleme e}
\left\{ \begin{array}{l}
       P\tilde{v}=e v_0 \\
       \tilde{v} = 0\\
       \partial_{\nu} \tilde{v} -i a \partial_{x_0} \tilde{v}+b \tilde{v} = v_1
     \end{array} \begin{array}{l}
       \text{ in } X,\\
       \text{ on } \partial X_D,\\
       \text{ on } \partial X_N,
     \end{array} \right.
\end{equation}
where, as in \cite{LR2}, we have used the notation $b =- \partial_{\nu} (1 /e)_{|\partial X}$.
\newline \noindent To localize and use the local form of our operator, we choose some  cut-off function $\theta$ with sufficiently small compact support. Let
$g =\theta \tilde{v}$ satisfying different problems, depending on the localization chosen.
\newline The first problem is a problem without boundary conditions (if the support of the cut-off is away from $\partial X$).
Moreover, if the support of the cut-off function intersects the boundary, we have three different cases to consider: one where the only boundary condition 
is of Dirichlet type, one where the only boundary condition is of Neumann type and the last one is a Zaremba boundary problem.
\vspace{0.2 cm}
\newline In each situation, we need some adapted Carleman estimates. In the three first situations, these results were obtained by Lebeau and Robbiano: it is Proposition 2 of \cite{LR1} and Proposition 1, Proposition
2 of \cite{LR2}, recalled below.
\newline
\par As before, we consider $B_\kappa:=\{x\in \R^{d+1};|x|\le \kappa\}$ and $\VP$ a $\Cinf$ function. We note $\Cinf_0(B_\kappa)$ the set of $\Cinf$ functions supported in $B_\kappa$.
\begin{proposition}\label{carlemansansconditionbord}
  Assume that \eqref{nonzero}, \eqref{hypohormander} hold. 
  Then, there exists $C, h_0 > 0$ such that for any $h \in (0, h_0)$ and for
  any $g \in \Cinf_0 (B_\kappa)$, the following inequality holds
 $$ \| g e^{\varphi / h}\|_{L^2(x_d>0)}  +\| h(\partial_{x} g) e^{\varphi / h}\|_{L^2(x_d>0)} \le C\left(h^{-1/2} \|h^2 P(g) e^{ \varphi / h}\|_{L^2(x_d>0)}+  |g e^{ \varphi / h}|+  |h (\partial_{x_d} g) e^{\varphi / h})|)\right).$$
\end{proposition}

\begin{proposition}\label{carlemandirichlet}
  Assume that \eqref{nonzero}, \eqref{hypohormander} hold and that
  \[ \frac{\partial \varphi}{\partial x_d} > 0 \text{ on } \{x_d=0\} \cap B_\kappa. \]
  Then there exists $C, h_0 > 0$ such that for any $h \in (0, h_0)$, $g_0\in \Cinf(\{x_d=0\})$, and for
  any $g \in \Cinf_0 (B_\kappa)$ such that
  \[ g = 0 \text{ on } \{x_d=0\}, \]
  the following inequality holds:
 $$\| g e^{\varphi / h}\|_{L^2(x_d>0)}  +\| h(\partial_{x} g) e^{\varphi / h}\|_{L^2(x_d>0)} \le C h^{-1/2} \|h^2 P(g) e^{ \varphi / h}\|_{L^2(x_d>0)}.$$
\end{proposition}

\begin{proposition}\label{carlemanneumann}
  Assume that $a$ is smooth, \eqref{nonzero}, \eqref{hypohormander} hold and that the following hypotheses are fulfilled
  \[ \frac{\partial \varphi}{\partial x_d} > 0 \text{ on } \{x_d=0\} \cap B_\kappa, \]
  \[ 1 > a^2 \text{ on } \{x_d=0\} \cap B_\kappa, \]
  \[ (1 - a^2) \left( \frac{\partial \varphi}{\partial x_d} \right)^2
     > a^2 \left[ r \left( x, \frac{\partial \varphi}{\partial x'}
     \right) - a^2 \left( \frac{\partial \varphi}{\partial x_0}
     \right)^2 \right] \text{ on } \{x_d=0\} \cap B_\kappa. \]
  Then there exists $C, h_0 > 0$ such that for any $h \in (0, h_0)$, $g_1\in \Cinf(\{x_d=0\})$, and
  any $g \in \Cinf_0 (B_\kappa)$ such that
  \[ \partial_{x_d} g + i a \partial_{x_0} g - b g = g_1 \text{ on } \{x_d=0\}, \]
  the following inequality holds:
$$\| g e^{\varphi / h}\|_{L^2(x_d>0)}  +\| h(\partial_{x} g) e^{\varphi / h}\|_{L^2(x_d>0)}+| h(\partial_{x'} g) e^{\varphi / h}|
\le C\left(h^{-1/2}\|h^2 P(g) e^{ \varphi / h}\|_{L^2(x_d>0)}+ |h g_1 e^{ \varphi / h}|\right).$$
\end{proposition}

\vspace{0.2 cm}
\noindent However, due to the low regularity of the function $a$ here, we cannot apply this result directly and shall derive the following suitable result by a perturbation argument.
\begin{corollary}\label{carlemanneumann2}
 Assume that \eqref{nonzero} and \eqref{hypohormander} hold. If $a$ is a function such that $a(x) \tendvers{x}{0} 0$,
$$\frac{\partial \varphi}{\partial x_d} > 0 \text{ on } \{x_d=0\} \cap B_\kappa $$ 
then, there exists \ $\kappa, C, h_0 > 0$, such that, for any $h \in (0,
  h_0)$, $g_1\in L^{2}(\{x_d=0\})$ and any $g \in H^1 (\R^n)$ supported in $B_\kappa$ which satisfies
  \[ P(g)\in L^2(\R^n) \text{ and } \partial_{x_d} g +i a \partial_{x_0} g- b g = g_1  \text{ on }  \{x_d=0\} , \]
  the following inequality holds:
$$
\| g e^{\varphi / h}\|_{L^2(x_n>0)}  +\| h(\partial_{x} g) e^{\varphi / h}\|_{L^2(x_n>0)}\le C\left(h^{-1/2}\|h^2 P(g) e^{ \varphi / h}\|_{L^2(x_d>0)}+  |h g_1 e^{ \varphi / h}|\right).
$$
\end{corollary}
\begin{proof}
 We apply Proposition \ref{carlemanneumann} with $\overline{a}:=0$ and $\overline{g_1}:=g_1-i a \partial_{x_0} g$ to get, after a standard approximation argument,
\begin{eqnarray*}
\| g e^{\varphi / h}\|_{L^2(x_n>0)}  +\| h(\partial_{x} g) e^{\varphi / h}\|_{L^2(x_n>0)}+| h(\partial_{x'} g) e^{\varphi / h}|&\le& Ch^{-1/2}\|h^2 P(g) e^{ \varphi / h}\|_{L^2(x_d>0)}\\
&+& C \left(|h g_1 e^{ \varphi / h}| +|h a \partial_{x_0} g e^{ \varphi / h}|\right)\\
&\le& C h^{-1/2}\|h^2 P(g) e^{ \varphi / h}\|_{L^2(x_d>0)}\\
&+& C|h g_1 e^{ \varphi / h}| + h \ C  \sup_{B_\kappa} |a| \ |\partial_{x_0} g e^{ \varphi / h}|.
\end{eqnarray*}
Choosing now $\kappa$ sufficiently small so that $\displaystyle{C \sup_{B_\kappa} |a|} \leq 1/2$, we get the desired estimate.
\end{proof}

We finally deduce from Theorem \ref{Carleman-mixte} the analog of the Carleman estimates of Propositions \ref{carlemansansconditionbord}
, \ref{carlemandirichlet} and Cororally \ref{carlemanneumann2}.
\begin{corollary}\label{carlemanmelee}
Assume that \eqref{nonzero} and \eqref{hypohormander} hold. If $\rho>1/2$, $a\in \Con^{\rho} (\{x_n=0, \ x_1<0\})$ such that $a(x) \tendvers{x_1}{0} 0$,
$$  \left( \frac{\partial \varphi}{\partial x_d} > 0 \text{ on } \{x_d=0\} \cap B_\kappa \right)\text{ and }|\partial_{x'} \VP(0)| \le \eps \partial_{x_d} \VP(0)$$ 
then, for any $\eps$ sufficiently small, there exists \ $\kappa, C, h_0 > 0$, such that, for any $h \in (0,
  h_0)$, $g_1\in L^{2}(x_1<0)$ and any $g \in H^1 (\R^n)$ supported in $B_\kappa$ which satisfies
  \[ P(g)\in L^2(\R^n) \text{ and }\left\{ \begin{array}{l}
       g = 0\\
       \partial_{x_d} g +i a \partial_{x_0} g- b g = g_1
     \end{array} \begin{array}{l}
       \text{ on }  \{x_d=0, x_1>0\} ,\\
       \text{ on }  \{x_d=0, x_1<0\},
     \end{array} \right. \]
  the following inequality holds:
$$
\| g e^{\varphi / h}\|_{L^2(x_n>0)}  +\| h(\partial_{x} g) e^{\varphi / h}\|_{L^2(x_n>0)}\le C\left(h^{-1/2}\|h^2 P(g) e^{ \varphi / h}\|_{L^2(x_d>0)}+  |h g_1 e^{ \varphi / h}|_{L^{2}(x_1<0)}\right).
$$
\end{corollary}
\begin{proof} 
One applies Theorem \ref{Carleman-mixte} (with $n=d+1$) with $\overline{g_0}:=0$, $\overline{g_1}:=g_1-i a\partial_{x_0} g+b g$  
and $g$ such that
 \[ \left\{ \begin{array}{l}
       g = \overline{g_0}\\
       \partial_{x_d} g = \overline{g_1}
     \end{array} \begin{array}{l}
       \text{ on } \{x_d=0, x_1>0\},\\
       \text{ on } \{x_d=0, x_1<0\}.
     \end{array} \right. \]
We get that $\| g e^{\varphi / h}\|_{H^1(x_d>0)}  + | g e^{\varphi / h}|_{1/2} +|h (\partial_{x_d}g) e^{\varphi / h}|_{-1/2}$ is bounded by

\begin{eqnarray*}
 C \left(h^{-1/2}\|h^2 P(g) e^{ \varphi / h}\|_{L^2(x_d>0)}+ |h g_1 e^{ \varphi / h}|_{H^{-1/2}_{sc}(x_1<0)}\right)\\
+C\left(|h a (\partial_{x_0} g) e^{ \varphi / h}|_{H^{-1/2}_{sc}(x_1<0)}+|h b g e^{ \varphi / h}|_{H^{-1/2}_{sc}(x_1<0)}\right).
\end{eqnarray*}
Moreover, we have
\begin{equation}\label{perturbation1}|h b g e^{ \varphi / h}|_{H^{-1/2}(x_1<0)}\le|h b g e^{ \varphi / h}|_{L^{2}(x_1<0)}\le h  \sup_{B_\kappa }|b| \ |g e^{ \varphi / h}|_{H^{1}(x_1<0)}.\end{equation}
On the other hand, one has $(\partial_{x_0} g) e^{\VP/h} \in H^{-1/2}(x_1<0)$ so there exists $u\in H^{-1/2}(\R^{n-1})$ such that $u_{|x_1<0}=(\partial_{x_0} g) e^{\VP/h}$.\\
Denoting  $\overline{a}:=1_{\{x_1<0\}} a \in \Con^{\rho}(\R^{n-1})$, one now has, provided $\rho>1/2$,
\begin{equation}\label{continuite-1/2} \overline{a} u \in H^{-1/2}(\R^{n-1}),\end{equation}
using standard continuity result (see e.g. \cite[Corollary p. 143]{Triebel}).\\
Since $\overline{a} u =1_{\{x_1<0\}} a (\partial_{x_0} g) e^{ \varphi / h}=\overline{a}(\partial_{x_0} g) e^{ \varphi / h}$ and  according to the definition of the $H^{-1/2}_{sc}(x_1<0)$ norm, one consequently has 
$$| a (\partial_{x_0} g) e^{ \varphi / h}|_{H^{-1/2}_{sc}(x_1<0)}\le |\overline{a} (\partial_{x_0} g) e^{ \varphi / h}|_{-1/2}.$$
Using \cite[Corollary p. 143]{Triebel} again, we deduce, for $1/2<\rho'<\rho$,
\begin{equation}\label{continuite-1/2bis}| a (\partial_{x_0} g) e^{ \varphi / h}|_{H^{-1/2}_{sc}(x_1<0)}\le |\overline{a} (\partial_{x_0} g) e^{ \varphi / h}|_{-1/2} \le C\|\overline{a} \chi(./\kappa)\|_{\Con^{\rho'}} \ | (\partial_{x_0} g) e^{ \varphi / h}|_{H^{-1/2}_{sc}(x_1<0)},\end{equation}
where $\chi$ is some regular cut-off function supported in $B_2$ such that $\chi=1$ on $B_1$ and, for any $x\in \R^{n-1}$, $\chi(./\kappa)(x)=\chi(x/\kappa)$.\\
One may now apply Lemma \ref{holderpetit} and get that
\begin{equation}\label{perturbation2}| a (\partial_{x_0} g) e^{ \varphi / h}|_{H^{-1/2}_{sc}(x_1<0)}\le C \kappa^{\rho-\rho'} \|a\|_{\Con{^\rho}} | (\partial_{x_0} g) e^{ \varphi / h}|_{H^{-1/2}_{sc}(x_1<0)}.
\end{equation}
Summing up \eqref{perturbation1} and \eqref{perturbation2}, we deduce that, for $h$ and $\kappa$ sufficiently small,
\begin{eqnarray*}
& \| g e^{\varphi / h}\|_{H^1(x_d>0)}  + | g e^{\varphi / h}|_{1/2}  +  |h (\partial_{x'} g) e^{\varphi / h}|_{-1/2}&\\
&\le C \left(h^{-1/2}\|h^2 P(g) e^{ \varphi / h}\|_{L^2(x_d>0)}+  |h g_1 e^{ \varphi / h}|_{H^{-1/2}_{sc}(x_1<0)} \right).&
\end{eqnarray*}
Finally, writing $g=(g e^{\VP/h})e^{-\VP/h}$, we have
$$ \| g e^{\varphi / h}\|_{L^2(x_d>0)} + \|h (\partial_x g) e^{\varphi / h}\|_{L^2(x_d>0)}\le C\| g e^{\varphi / h}\|_{H^1(x_d>0)}$$
and the sought result is proved.
\end{proof}
\subsection{Proof of Proposition \ref{interpolation}}
We follow the method of Lebeau and Robbiano (see \cite[Paragraph 3.B]{LR1}).
We here simply insist on the points that differ in our context. 
\paragraph{Interpolation inequality away from the boundary}
Defining $X_d = (- 3 / 4, 3 / 4) \times \Omega_d$ with $\Omega_d =\{x \in
\Omega ; d (x, \partial \Omega) \geqslant d\}$, we first recall an
interpolation inequality for system \eqref{ondesabstrait} away from the boundary.

\begin{lemma}\label{interpolationDirichlet}
  There exists $C > 0$ and $\tau_0 \in (0, 1)$ such that for
  any $\tau \in [0, \tau_0]$ and for any function $v$ solution of \eqref{ondesabstrait}, the
  following inequality holds
  \[  \| v \|_{H^1 (X_d)} \leqslant C \left( \| \Delta_X v
     \|_{L^2 (X)} + \|v \|_{L^2 (\partial X_N^\delta)} + \| \partial_{\nu} v \|_{L^2 (\partial X_N^\delta)} \right)^{\tau} \| v\|^{1 - \tau}_{H^1 (X)}. \]
\end{lemma}
\begin{proof} We first have, for any function $v$ solution of \eqref{ondesabstrait},
\[  \| v \|_{H^1 (X_d)} \leqslant C \left( \| \Delta_X v
     \|_{L^2 (X)} + \| \theta v \|_{L^2 (\partial X)} + \| \theta
     \partial_{\nu} v \|_{L^2 (\partial X)} \right)^{\tau} \| v\|^{1 - \tau}_{H^1 (X)} \]
for $\theta \in \Cinf_0(\partial X)$ any non-trivial function.
\newline \noindent Indeed, if $v=0$ on $\partial X$, this result is contained in \cite[Paragraph 3.B, equation (45)]{LR1}.
The proof of this slightly more general interpolation result is
absolutely identical and we shall not detail it here (see also \cite[Paragraph 3.B, Lemme 3]{LR1} for a local version of this estimate).
\newline \noindent It suffices now to choose $\theta$ supported in $\partial X_N^\delta$ to obtain the Lemma.
\end{proof}
\paragraph{Interpolation method near the boundary}

We now use the Carleman estimates described in Section 3.1 to prove estimates
near the boundary. 

\par First of all, we begin by the definition of the phase function inspired by
\cite{LR1}. We put $\VP=f(\psi)$ where $f(t)=e^{\beta t}$  and, writing $x=(x_0,y)$,
\[ \forall x\in X, \ \psi (x) = \psi_0 (x_0)+\psi_1 (y)\]
where, for some $d>0$, $\psi_1$ is such that $\partial_{\nu} \psi_1 (y) > 0$ if $d (y, \partial
\Omega) < 3 d$ (for $\partial_{\nu}$ a vector field defined in some
neighbourhood of $\partial \Omega$ which extends the normal derivative on
$\partial \Omega$) and
\[ \psi_1 (y) = \left\{ \begin{array}{l}
     d(y,\partial \Omega)\\
     3 d
   \end{array} \begin{array}{l}
     \text{ if } d (y, \partial \Omega) \leqslant 2 d,\\
     \text{ if } d (y, \partial \Omega) \geqslant 3 d,
   \end{array} \right. \]
and, for some $\eps>0$, $\psi_0$ is an even function such that $\psi'_0 (x_0) < 0$ if $x_0 > 1/ 2$ with
\[ \psi_0(x_0) = \left\{ \begin{array}{l}
     0\\
     \eps (1  -2 x_0)
   \end{array} \begin{array}{l}
     \text{ if } x_0 \in [0, 1 / 2],\\
     \text{ if } x_0 \in [3 / 4, 1).
   \end{array} \right. \]
\par We will now show that we can apply our Carleman estimates to our function
$\varphi$.
\newline \noindent First, it is classical that the function $\varphi=f(\psi)$ satisfies H\"{o}rmander
hypoellipticity condition \eqref{hypohormander} for some $\beta > 0$ large enough. We refer the reader to Lemme 3 in \cite[Paragraph 3.B]{LR1} for a proof.
\newline \noindent On the other hand, since $\partial_{\nu}\VP=f'(\psi)\partial_{\nu}\psi $ and $\partial_{x'}\VP=f'(\psi) \partial_{x'}\psi $, it is clear that
$$|\partial_{x'}\VP|\le C \eps |\partial_{\nu}\VP| \text{ on } \partial X$$
and the hypotheses of Theorem \ref{Carleman-mixte} are fullfilled for $\eps$ sufficiently small.
\par A finite partition of unity on $\partial \Omega$ combined
with Propositions \ref{carlemansansconditionbord}, \ref{carlemandirichlet} and Corollaries \ref{carlemanneumann2}, \ref{carlemanmelee} may now show that
\begin{eqnarray}\label{carlemanglobal}
& \| e^{\varphi / h} w
  \|^2_{L^2 (X)} + \| e^{\varphi / h} h \partial_x w \|^2_{L^2 (X)}&\\
& \le C \left(h^3 \| e^{\varphi / h} P w \|^2_{L^2 (X)} + h^2 \| e^{\varphi / h}
 (\partial_{\nu} - i a \partial_{x_0}+b) w \|^2_{L^2 (\partial X_N)} \right.&\nonumber \\
& \left. + \|e^{ \varphi / h}w\|_{L^2(\partial X_N^\delta)} + \|he^{ \varphi / h}\partial_\nu w\|_{L^2(\partial X_N^\delta)} \right)\nonumber &
\end{eqnarray}
for any $w \in H^1(X)$ supported in some small neighbourhood $W$ of $\partial X$ which also satisfies $P w\in L^2(X)$, $w = 0$ on $\partial X_D$ and $\partial_\nu w\in L^2(\partial X_N)$. 
\newline \noindent Indeed, one first chooses the
partition of unity $(\theta_i)$ on some \nhd of $\partial \Omega$ such that any element of this partition $\theta$ lies in one of the following cases:
\begin{enumerate}
\item $\supp (\theta) \cap \partial \Omega \subset \partial \Omega_D$,
\item $ \supp (\theta) \cap \partial \Omega \subset \{y \in \partial \Omega_N
   ; a (y) < 2 \delta\}$,
\item $\supp (\theta) \cap \partial \Omega \subset \{y \in \partial \Omega_N
   ; a (y) > \delta\}$,
\item $\supp (\theta) \cap \partial \Omega \subset \partial \Omega_D \cup
   \{y \in \partial \Omega_N ; a (y) < 2 \delta\}$ and $\theta$ supported in a \nhd of $\Gamma$.
\end{enumerate}
Next, for $\delta$ and $\supp(\theta)$ chosen sufficiently small, one defines $g=\theta \tilde{v}$. 
Working in local coordinates such that \eqref{bord} and \eqref{borddubord} hold, we may apply to function $g$
\begin{itemize}\item Proposition \ref{carlemandirichlet} in case 1, 
 \item Corollary \ref{carlemanneumann2} in case 2, 
\item Proposition \ref{carlemansansconditionbord} in case 3, 
\item Corollary \ref{carlemanmelee} in case 4, \end{itemize}
and, summing up these inequalities, we directly get the estimate \eqref{carlemanglobal}. Note in particular that the estimates of Propositions 
\ref{carlemandirichlet}, \ref{carlemanneumann} and \ref{carlemansansconditionbord} can be applied to $w$, using a standard approximation argument.
\par To get now our interpolation inequality, we
define, for $r_1 < r'_1 < r_2 < r_2' < r_3 < r'_3$, the sets (see Figure 2) 
$$
  V =\{x \in X ; r_1 \leq \psi (x) \leq r'_3 \}$$
and, for $j=1,2,3$,
$$  V_j =\{x \in X ; r_j \leq \psi(x) \leq r_j' \} . $$
As in \cite{LR1}, we choose $r_1 = - 2 d$, $r'_1 =  - d$, $r_2 = 0$, $r'_2
= d$, $r_3 = 3 / 2  d$, $r'_3 = 2  d$ so that,
using the definition of $\psi$,
\begin{equation}\label{inclusions}
  (- 1 / 2, 1 / 2) \times (\Omega \backslash \Omega_d) \subset V_2,\qquad V_3
  \subset (- 1, 1) \times \Omega_d . 
\end{equation}
Moreover, using that $\inf \psi_0=-\eps$ and $\sup \psi_1=3d$, one has $V \varsubsetneq X$ for $d$ sufficiently small so that $ - 5 d > - \eps$.

\begin{figure}[h!]
\centerline{\begin{picture}(0,0)%
\includegraphics{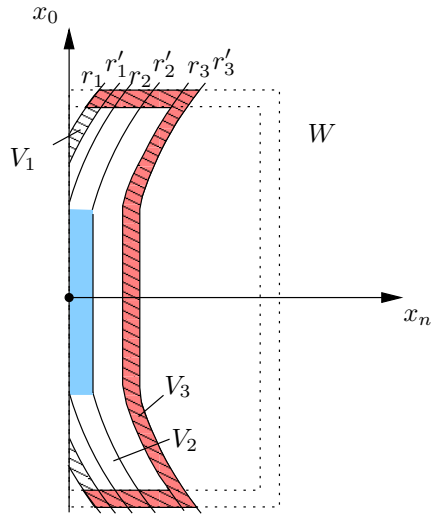}%
\end{picture}%
\setlength{\unitlength}{3947sp}%
\begingroup\makeatletter\ifx\SetFigFontNFSS\undefined%
\gdef\SetFigFontNFSS#1#2#3#4#5{%
  \reset@font\fontsize{#1}{#2pt}%
  \fontfamily{#3}\fontseries{#4}\fontshape{#5}%
  \selectfont}%
\fi\endgroup%
\begin{picture}(2505,3260)(1711,-2475)
\put(2701,-1716){\makebox(0,0)[lb]{\smash{{\color[rgb]{0,0,0}$V_3$}
}}}
\put(2742,-2050){\makebox(0,0)[lb]{\smash{{\color[rgb]{0,0,0}$V_2$}
}}}
\put(1871,638){\makebox(0,0)[lb]{\smash{{\color[rgb]{0,0,0}$x_0$}
}}}
\put(3000,326){\makebox(0,0)[lb]{\smash{{\color[rgb]{0,0,0}$r_3'$}
}}}
\put(2845,296){\makebox(0,0)[lb]{\smash{{\color[rgb]{0,0,0}$r_3$}
}}}
\put(3601,-136){\makebox(0,0)[lb]{\smash{{\color[rgb]{0,0,0}$W$}
}}}
\put(4201,-1261){\makebox(0,0)[lb]{\smash{{\color[rgb]{0,0,0}$x_n$}
}}}
\put(1726,-286){\makebox(0,0)[lb]{\smash{{\color[rgb]{0,0,0}$V_1$}
}}}
\put(2476,239){\makebox(0,0)[lb]{\smash{{\color[rgb]{0,0,0}$r_2$}
}}}
\put(2176,239){\makebox(0,0)[lb]{\smash{{\color[rgb]{0,0,0}$r_1$}
}}}
\put(2626,314){\makebox(0,0)[lb]{\smash{{\color[rgb]{0,0,0}$r_2'$}
}}}
\put(2326,314){\makebox(0,0)[lb]{\smash{{\color[rgb]{0,0,0}$r_1'$}
}}}
\end{picture}%
}
\caption{Interpolation sets $V_1$, $V_2$, $V_3$ and level sets associated to $(r_i)_{1\le i\le 3}$, $(r'_i)_{1\le i\le 3}$ .}
\end{figure}

\par  We now apply the interpolation method (see Lemme 3 in \cite[Paragraph 3.B]{LR1}).
We choose $\chi$ a $\Cinf$ function supported in $V$ and $W$ such that
$\chi = 1$ for $\psi (x) \in [r'_1, r_3]$ and we apply the Carleman estimate \eqref{carlemanglobal} to $w = \chi \tilde v$ where $\tilde v$ is the solution of \eqref{probleme e}.

\noindent We write $P w = \chi P \tilde v + [P, \chi] \tilde v$ and we note that $[P, \chi]$ is a differential operator of order $1$ supported in $V_1 \cup V_3$.
Hence,
$$\|e^{\VP/h} P w\|_{L^2( X)}\le \|e^{\VP/h} v_0\|_{L^2(V)}+C\|e^{\VP/h} \tilde v\|_{H^1(V_1)}+C\|e^{\VP/h} \tilde v\|_{H^1(V_3)}.$$
Analogously, using trace estimates on $V_1$ and $V_3$, we have
$$\|e^{\VP/h} (\partial_\nu -i a\partial_{x_0}+b)w\|_{L^2(\partial X_N)}\le \|e^{\VP/h} v_1\|_{L^2(\partial X_N \cap V)}+C\|e^{\VP/h} \tilde v\|_{H^1(V_1)}+C\|e^{\VP/h} \tilde v\|_{H^1(V_3)}.$$
Summing up, we obtain by $\eqref{carlemanglobal}$
\begin{eqnarray*}
 \| e^{\varphi / h} \tilde v \|_{L^2 (V_2)} + \|
   e^{\varphi / h} \partial_x \tilde v \|_{L^2 (V_2)} &\leqslant& C \left( \|
   e^{\varphi / h} v_0 \|_{L^2 (V)} + \| e^{\varphi / h}
   v_1 \|_{L^2 (\partial X_N\cap V)} \right.\\
&+&   \| e^{\varphi / h} \tilde v \|_{L^2 (\partial
   X^{\delta}_N\cap V)}+\| e^{\varphi / h} \partial_{\nu} \tilde v
   \|_{L^2 (\partial X^{\delta}_N\cap V)} \\
&+& \| e^{\varphi / h} \tilde v
   \|_{L^2 (V_1)} + \| e^{\varphi / h} \partial_x \tilde v \|_{L^2 (V_1)} \\
&+& \left. \| e^{\varphi / h} \tilde v \|_{L^2 (V_3)} + \| e^{\varphi / h} \partial_x \tilde v
   \|_{L^2 (V_3)} \right)
 \end{eqnarray*}
so that, using the definition of $\VP$,
\begin{eqnarray*}
&e^{f(r_2) / h} \| \tilde v \|_{H^1 (V_2)} \leqslant C e^{f(r'_1) / h}\|  \tilde v \|_{H^1 (V_1)}&\\
&+ C e^{f(r'_3) / h}\left( \| v_0 \|_{L^2 (X)} + \| v_1 \|_{L^2 (\partial X_N)}+\| \tilde v \|_{L^2 (\partial X^{\delta}_N)}+\|\partial_{\nu} \tilde v \|_{L^2 (\partial X^{\delta}_N)} + \|  \tilde v \|_{H^1 (V_3)} \right)&\\
 \end{eqnarray*}
and, using \eqref{inclusions} and coming back to the solution $v$ of \eqref{ondesabstrait}, 
\begin{eqnarray*}
&e^{f(r_2) / h} \|v \|_{H^1 (Y/X_d)} \leqslant C e^{f(r'_1) / h}\|v \|_{H^1 (X)}&\\
&+ C e^{f(r'_3) / h}\left( \| v_0 \|_{L^2 (X)} + \| v_1 \|_{L^2 (\partial X_N)}+\|v \|_{L^2 (\partial X^{\delta}_N)}+\|\partial_{\nu} v \|_{L^2 (\partial X^{\delta}_N)} + \|v \|_{H^1 (X_d)} \right).&\\
 \end{eqnarray*}
In the same way as it was done in \cite[Paragraph 3.B]{LR1} and since $f(r'_1)<f(r_2)<f(r'_3)$, one may deduce after an optimization in $h\in(0,h_0)$ that there exists $\tau_0\in(0,1)$ such that
\begin{eqnarray*}& \| v  \|_{H^1 (Y \backslash X_d)}&\\
 &\leqslant C \left(
   \|v_0 \|_{L^2 (X)} + \|v_1 \|_{L^2 (\partial X_N)} +
   \| v \|_{L^2 (\partial X^{\delta}_N)} + \|
   \partial_{\nu} v \|_{L^2 (\partial X^{\delta}_N)} +
   \| v \|_{H^1 (X_d)} \right)^{\tau_0} \| v  \|_{H^1 (X)}^{1-\tau_0}.&
\end{eqnarray*}
If one combines this result with Lemma \ref{interpolationDirichlet} and taking account that
$\partial_{\nu}v=-ia(x)\partial_{x_0}v+v_1$ on $\partial X_N^\delta$, we get the sought result of Proposition \ref{interpolation}.

\subsection{End of the proof}

\subsubsection{ Preliminary settings}

First of all, let us recall that the system \eqref{ondesmelees} possesses a unique solution.
We define the Hilbert space $H = H^1_D (\Omega) \times L^2 (\Omega)$ (where
$H^1_D (\Omega) =\{u \in H^1 (\Omega) ; u = 0 \text{ on } \partial \Omega_D \}$)
equipped with the norm
\[ \left. \| (u_0, u_1) \right\|^2_H = \int_{\Omega} | \partial_{x} u_0 |^2 + |u_1
   |^2 d x. \]
We also define the unbounded operator $\mathcal{A} $ on $H$ by
\[ \mathcal{A}  = \left( \begin{array}{l}
     0\\
     \Delta
   \end{array} \begin{array}{l}
     I\\
     0
   \end{array} \right) \]
with domain
\[ \mathcal{D}(\mathcal{A} ) =\{u = (u_0, u_1) \in H ; \mathcal{A}  u \in H \text{ and }
   \partial_{\nu} u_0 + a u_1 = 0\}. \]
It is clear that system \eqref{ondesmelees} can be rewritten in terms of the abstract problem
\[ \left\{ \begin{array}{l}
     \partial_t U (t) = \mathcal{A}  U (t),\\
     U (0) = \left( \begin{array}{l}
       u_0\\
       u_1
     \end{array} \right),
   \end{array} \right.  \]
where $U = \left( \begin{array}{l}
  u\\
  \partial_t u
\end{array} \right)$.
\newline \noindent Moreover, $\mathcal{A} $ is a monotone operator. Indeed, an integration by parts gives,
for any $u = \left( \begin{array}{l}
  u_0\\
  u_1
\end{array} \right) \in \mathcal{D}(\mathcal{A} )$,
\[ \Re (\mathcal{A}  u, u)_H = \int_{\partial \Omega_N} \overline{\partial_{\nu}
   u_0} u_1 d\sigma = - \int_{\partial \Omega_N} a |u_1 |^2 d\sigma \leqslant 0. \]
On the other hand, $\mathcal{A}  - I$ is an isomorphism from $\mathcal{D(A)}$ to $H$. Indeed, one easily obtains for
$\left( \begin{array}{l}
  f_0\\
  f_1
\end{array} \right) \in H$,
\[ (\mathcal{A}  - I) \left( \begin{array}{l}
     u_0\\
     u_1
   \end{array} \right) = \left( \begin{array}{l}
     f_0\\
     f_1
   \end{array} \right) \Leftrightarrow \left\{ \begin{array}{l}
     - u_0 + \Delta u_0 = f_0 + f_1,\\
     u_1 = u_0 + f,
   \end{array} \right. \]
and the bijectivity is granted by Lax-Milgram noting that the bilinear form
\[ (u, v) \in H_D^1 (\Omega)^2 \mapsto \int_{\Omega} (\partial_{x} u. \partial_{x} \bar{v}  +
   u \bar{v}) d x+ \int_{\partial \Omega_N} a u \bar{v} d\sigma \]
is coercive. Hence, Hille-Yoshida theorem gives that $\mathcal{A} $ generates a strongly
continuous semigroup on $H$ and hence the existence and unicity of solutions to
problem \eqref{ondesmelees}.

\subsubsection{ Proofs of Proposition \ref{spectral} and Theorem \ref{decroissancelog}}

We consequently focus on the equation
\[ (\mathcal{A}  -i \lambda I) \left( \begin{array}{l}
     u_0\\
     u_1
   \end{array} \right) = \left( \begin{array}{l}
     f_0\\
     f_1
   \end{array} \right) . \]
We will write $R (\mu) = (\mathcal{A}  - \mu I)^{- 1}$ when it is defined.

\noindent One has the following system
\begin{equation}\label{laplacienspectral}
  \left\{ \begin{array}{l}
    (\Delta + \lambda^2) u_0 = (i \lambda f_0 + f_1)\\
    (\partial_{\nu} + i \lambda a (x)) u_0 = - a (x) f_0\\
    u_0 = 0
  \end{array} \begin{array}{l}
    \text{ in } \Omega,\\
    \text{ on } \partial \Omega_N,\\
    \text{ on } \partial \Omega_D,
  \end{array} \right.
\end{equation}
 and  we introduce
$v (x_0, x) = e^{  \lambda x_0} u_0(x)$, $$X = (- 1, 1) \times \Omega, \partial X_N = (- 1, 1) \times \partial \Omega_N \text{ and } \partial X_D = (- 1,
1) \times \partial \Omega_D$$ so that $v$ is solution of
\begin{equation*}
\left\{ \begin{array}{l}
     \Delta_X  v = e^{\lambda x_0} (i\lambda f_0+f_1) \\
     (\partial_{\nu}  + i a (x) \partial_{x_0}) v = e^{\lambda x_0} (-a(x) f_0)\\
     v = 0
   \end{array} \begin{array}{l}
     \text{ in } X,\\
     \text{ on } \partial X_N,\\
     \text{ on } \partial X_D.
   \end{array} \right.
\end{equation*}
Applying Proposition \ref{interpolation} to $v_0:= e^{\lambda x_0} (i\lambda f_0+f_1)$ and $v_1:=e^{\lambda x_0} (-a(x) f_0)$, 
one may get the estimate, for  any $\lambda \in \R$,
\begin{equation}\label{continuation unique}
\| u_0 \|_{H^1 (\Omega)} \leqslant C e^{C| \lambda |} \left( \|
   f_0\|_{H^1 (\Omega)} + \| f_1 \|_{L^2 (\Omega)} + \| u_0
   \|_{L^2 (\partial \Omega^{\delta}_N)} \right), 
\end{equation}
where we have noted $\partial \Omega_N^{\delta}=\{x\in \partial \Omega_N; a(x)>\delta\}$.
\newline If $\| u_0 \|_{L^2 (\partial \Omega^{\delta}_N)} \leqslant \|
f_0 \|_{H^1 (\Omega)} + \| f_1 \|_{L^2 (\Omega)}$, one
consequently has
\begin{equation}\label{but}
  \| u_0 \|_{H^1 (\Omega)} \leqslant C e^{C| \lambda |} \left( \|
  f_0 \|_{H^1 (\Omega)} + \| f_1 \|_{L^2 (\Omega)} \right).
\end{equation}
Else, we have the following estimate
\begin{equation}\label{estimationfausse2}
  \| u_0 \|_{H^1 (\Omega)} \leqslant C e^{C| \lambda |} \| u_0 \|_{L^2  (\partial \Omega^{\delta}_N)} .
\end{equation}
Using that
\[ \int_{\Omega} \overline{u_0} (- \Delta - \lambda^2) u_0 d x = - \lambda^2 \| u_0
   \|_{L^2 (\Omega)} + \int_{\Omega} | \partial_{x} u_0 |^2 d x - \int_{\partial
   \Omega_N} \overline{u_0} \partial_{\nu} u_0 d \sigma \]
and, since $(\Delta + \lambda^2) u_0 = i \lambda f_0 + f_1$ and $\partial_{\nu} u_0 +
i a \lambda u_0 = f_0$, we get, taking the imaginary part of this identity,
\[ | \lambda | \int_{\partial \Omega_N} a|u_0 |^2 d\sigma\leqslant (| \lambda |
   \| f_0 \|_{H^1 (\Omega)} + \| f_1 \|_{L^2 (\Omega)}) \| u_0
   \|_{L^2 (\Omega)} . \]
Consequently \eqref{estimationfausse2} and Young inequality give us
\[ \| u_0 \|_{H^1 (\Omega)} \leqslant C e^{C|\lambda |} ( \| f_0 \|_{H^1 (\Omega)} + \| f_1
   \|_{L^2 (\Omega)} ) \]
and one also gets the estimate \eqref{but}.

\noindent Since $u_1 = f_0 + i \lambda u_0$, we get in both cases the aditional estimate
\begin{equation*}
   \| u_1 \|_{L^2 (\Omega)} \leqslant \| f_0 \|_{H^1
  (\Omega)} + | \lambda | \| u_0 \|_{H^1 (\Omega)} \leqslant C e^{C| \lambda |}
  \left( \| f_0 \|_{H^1 (\Omega)} + \| f_1 \|_{L^2 (\Omega)}
  \right).
\end{equation*}
We consequently have proved that, for any $\lambda$
sufficiently large,
\[ \| u_0 \|_{H^1 (\Omega)} + \| u_1
   \|_{L^2 (\Omega)} \leqslant C e^{C| \lambda |} \left( \|
   f_0 \|_{H^1 (\Omega)} + \| f_1 \|_{L^2 (\Omega)} \right) \]
and that $\mathcal{A}  - i\lambda I$ is one-to-one. 
\newline \noindent Since $\mathcal{D}(\mathcal{A} ) \hookrightarrow H$ is compact, the Fredholm alternative (see e.g. \cite[Th\'eor\`eme VI.6]{B2} ) gives us that $\mathcal{A} $ is onto (because $(\mathcal{A}  -i \lambda
I) u = f \Leftrightarrow (I + (i \lambda - 1) R (1)) u = R (1) f$ where $R (1)$ is a compact operator) and finally that $\mathcal{A}  - i \lambda I$ is an
isomorphism. Moreover,
\[| \lambda | \geqslant \lambda_1 \Rightarrow \| R (i
   \lambda) \|_{H \rightarrow H} \leqslant C e^{C| \lambda |} . \]
On the other hand, the spectrum of $\mathcal{A}$ is discrete.
Furthermore, for any $\lambda \in \R$, $\mathcal{A}  - i \lambda I$ is one-to-one. 
\newline \noindent Indeed, if 
\[ (\mathcal{A}  -i \lambda I) \left( \begin{array}{l}
     u_0\\
     u_1
   \end{array} \right) = 0  \]
then, using a special case of \eqref{laplacienspectral}, one has
$$ \int_\Omega (\Delta+\lambda^2)u_0 \overline{u_0} d x=0$$
so that, using an integration by parts and the boundary conditions of $u_0$, 
$$ \int_\Omega (-|\partial_x u_0|^2+\lambda^2 |u_0|^2) d x - i \lambda \int_{\partial\Omega_N} a |u_0|^2 d \sigma=0.$$
Taking the imaginary part of this expression, we get $\lambda \times a u_0=0$ in $\partial\Omega_N$ and then $u_0=0$ on $\partial \Omega_N^\delta$. 
\\Since $u_0$ also satisfies system \eqref{laplacienspectral} with $f_0=0$ and $f_1=0$, the unique continuation estimate \eqref{continuation unique} gives us $u_0=0$ and hence $(u_0,u_1)=0$.

Using again that $\mathcal{D}(\mathcal{A} )
\hookrightarrow H$ is compact and the Fredholm alternative, we get that $i\mathbb{R} \subset \rho
(\mathcal{A} )$. 
\newline \noindent Since $\lambda \in \rho (\mathcal{A} ) \mapsto R (\lambda)$ is continuous,
there consequenty exists some constant $C > 0$ so that
\[ | \lambda | \leqslant \lambda_1 \Rightarrow \| R (i
   \lambda) \|_{H \rightarrow H} \leqslant C e^{C| \lambda |},\]
which concludes our proof of Proposition \ref{spectral}. 
\newline \noindent Theorem \ref{decroissancelog} is then an immediate application of \cite[Theorem A]{BD}  (see also Th\'eor\`eme 3 in \cite{Bu}).
\section{Comments and further applications}
\paragraph{Regularity of the function $a$.} We begin by some considerations on the sufficient regularity of the function $a$.
Using paradifferential calculus, is it not hard to get that $a$ can be chosen in the Besov space
$B^{1/2}_{\infty,2}(\partial \Omega_N) \hookrightarrow L^\infty(\partial \Omega_N)$. Indeed, the main point is to get continuity estimate analogous to Equations \eqref{continuite-1/2}, \eqref{continuite-1/2bis}
in the proof of Corollary \ref{carlemanmelee} and this can be done in a standard way by the use of paraproduct decomposition (introduced in \cite{Bony}).\\
Moreover, a regularity result of Shamir  on the solutions of the Zaremba problem allow us to extend our Theorem \ref{decroissancelog}.
More precisely, \cite{Sh} yields that, given $f\in L^2(\Omega)$, $f_0\in L^2(\partial \Omega_D)$, $f_1\in L^2(\partial \Omega_N)$, the solution $u\in H^1(\Omega)$ of the system
\begin{equation*}
  \left\{ \begin{array}{l}
\Delta u = f\\
    u = f_0\\
    \partial_{\nu} u = f_1\\
  \end{array} \begin{array}{l}
    \text{ in } \Omega, \\
    \text{ on } \partial \Omega_D,\\
    \text{ on } \partial \Omega_N,\\
  \end{array} \right.
\end{equation*}
belongs in fact to $H^{3/2-\epsilon}(\Omega)$ for any $\epsilon>0$. 
Some additional calculations may show that the result of Theorem \ref{decroissancelog} result remains valid for $a\in\Con^{\rho}(\partial \Omega_N)$ with $\rho>0$ (and even for $a\in B^{0}_{\infty,2}(\partial \Omega_N)$).

\paragraph{Other geometric cases.}

On the other hand, the Carleman estimate given in Theorem~\ref{Carleman-mixte} is local in a \nhd of  $(-1,1)\times \Gamma$, thus we can use it in other geometric cases, patching this estimate with other Carleman estimates
either to prove a global Carleman estimate or to prove an interpolation inequality in the spirit of Proposition~\ref{interpolation}.

 In particular, Theorem \ref{decroissancelog} remains valid if $\Omega$ is replaced by $(M,g)$ a smooth compact riemannian manifold with boundary and 
$\Gamma$ is a smooth submanifold of $\partial M$ of codimension 1. Note that $\Gamma$ is not necessarily connected (see Figure 3).
\newline \noindent Indeed, it is sufficient to work on each connected component of $M$ and to notice that our proof remains valid on each component.

\begin{figure}[h!]
\centerline{
\begin{picture}(0,0)%
\includegraphics{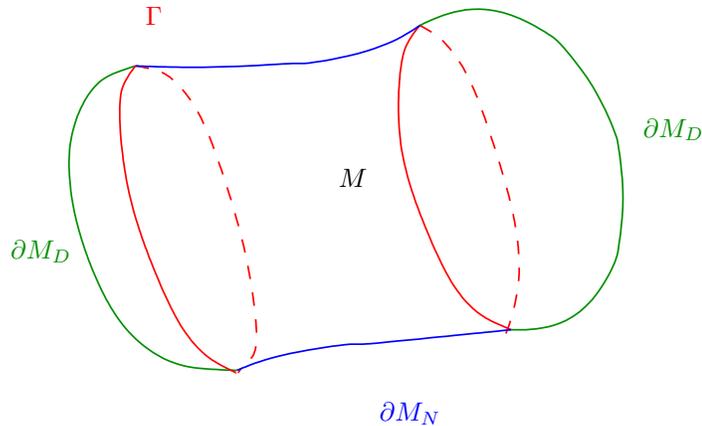}%
\end{picture}%
\setlength{\unitlength}{5594sp}%
\begingroup\makeatletter\ifx\SetFigFont\undefined%
\gdef\SetFigFont#1#2#3#4#5{%
  \reset@font\fontsize{#1}{#2pt}%
  \fontfamily{#3}\fontseries{#4}\fontshape{#5}%
  \selectfont}%
\fi\endgroup%
\begin{picture}(2834,1975)(2132,-2159)
\put(2147,-1368){\makebox(0,0)[lb]{\smash{{\color[rgb]{0,.56,0}$\partial M_D $}}}}
\put(4951,-826){\makebox(0,0)[lb]{\smash{{\color[rgb]{0,.56,0}$\partial M_D $}}}}
\put(2746,-331){\makebox(0,0)[lb]{\smash{{\color[rgb]{1,0,0}$\Gamma $}}}}
\put(3601,-1051){\makebox(0,0)[lb]{\smash{{\color[rgb]{0,0,0}$M $}}}}
\put(3781,-2086){\makebox(0,0)[lb]{\smash{{\color[rgb]{0,0,1}$\partial M_N $}}}}
\end{picture}%
}
\caption{A configuration example where $\Gamma$ is not connected.}
\end{figure}

\paragraph{Control of the heat equation.}

With the Carleman estimate found (Theorem \ref{Carleman-mixte}), we can also prove some null controllability results for the heat equation with the Zaremba boundary condition.
We give the result without detailled proof. 

Let $\omega$ an open subset of $\Omega$ such that $\overline{\omega}\neq \Omega$ and $T>0$. Then, for all $u_0\in L^2(\Omega)$, there exists $f\in L^2((0,T)\times\omega)$ such that the solution $u$ of the following system 

\begin{equation*}
  \left\{ \begin{array}{l}
\d_tu-\Delta u = 1_{\omega} f \\
    u = 0 \\
    \partial_{\nu} u = 0 \\
u_{|t=0}=u_0
 \end{array} \begin{array}{l}
    \text{ in } \Omega \times \mathbb{R}^+,\\
    \text{ on } \partial \Omega_D \times \mathbb{R}^+,\\
    \text{ on } \partial \Omega_N \times \mathbb{R}^+,\\
    \text{ in } \Omega,
  \end{array} \right.
\end{equation*}
satisfies $u(T,.)=0$.

We refer the reader to the survey by Le Rousseau and Lebeau \cite[Sections 5-6]{LL} where a strategy of proof is explained in details.
\par First of all, it is sufficient to have the following interpolation inequality (see \cite[Theorem 5.3]{LL}). 
\newline Let $X=(-1,1)\times\Omega$ and $Y=(-1/2,1/2)\times\Omega$. Then, there exists $C>0$ and $\tau\in(0,1)$ such that, for any $v\in H^2(X)$ such that $v=0$ on $(-1,1)\times\partial\Omega_D$,
$\partial_\nu v=0$ on $(-1,1)\times\partial\Omega_N$ and $v(-1,.)=0$ on $\Omega$,
$$
\|v\|_{H^1(Y)}\le C \|v\|_{H^1(X)}^{1-\tau}
\left(
\|(\partial_{x_0}^2+\Delta) v \|_{L^2(X)}+\|\partial_{x_0}v(0,.)\|_{L^2(\omega)} \right) ^\tau.
$$
This interpolation estimate is analogous but different from the one of Proposition~\ref{interpolation}. Nevertheless, we can prove it
 in the same way: the main novelty is to obtain a local interpolation inequality in a \nhd of $(-1/2,1/2)\times \Gamma$ and this can be done exactly as in Section 3.2.
\par With this interpolation inequality and following \cite[Theorem 5.4]{LL}, we can then obtain an inequality on sums of eigenfunctions. 
\newline \noindent Let $(\phi_j)_{j\in \N^*}$ be a orthonormal basis of eigenfunctions of the Laplacian with the Zaremba boundary condition and $0<\mu_1\le \mu_2\le ...$
the associated eigenvalues. 
Then, there exists $C>0$ such that, for all complex sequences $(\alpha_j)_{j\in\N^*}$ and all $\mu>0$,
$$
\sum_{\mu_j\le \mu}|\alpha_j|^2\le Ce^{C\sqrt{\mu}}\int_\omega
\Big|
\sum_{\mu_j\le \mu}\alpha_j\phi_j
\Big|^2dx.
$$
Following the strategy of  \cite[Section 6]{LL}, one is now able to construct a control function $f$.

%
%
\appendix

\section{Symbolic Calculus}
\label{Appendix A}

We use the metric definition given by H\"{o}rmander and, if $g$ is a metric and $m$ a $g$-continuous function, 
the notation $S(m,g)$ of symbol spaces (see definitions 18.4.1, 18.4.7 and 18.4.2 in \cite{Hoermander:V3}).
\newline \noindent We denote $\est{\eta}^2_\eps=|\eta|^2+\eps^2$ and $g$ 
the metric $$g=dx'^2+\frac{d\xi_1^2}{\este{\xi'}^2}+\frac{d\xi''^2}{\este{\xi''}^2}.$$

\begin{lemma}
\label{lem: bonne metrique bon poids}
$g$ is a metric slowly varying, semi-classical $\sigma$-temperate, uniformly with respect to $\eps$ if $h\le \eps $. 
\newline The weights $\est{\xi'}$, $\este{\xi'}$ and $\este{\xi''}$ are $g$-continuous and semi-classical $\sigma, g$-temperate uniformly with respect to $\eps$ if $h\le \eps $. 
\newline Moreover, we have $$\mathfrak{h}^2(x',\xi'):=\sup_{(y',\eta')}\left(\frac{g_{(x',\xi')}(y',\eta')}{g^\sigma_{(x',\xi')}(y',\eta')}\right)=\este{\xi''}^{-2}.$$
\end{lemma}
\vspace{0.8cm}
\begin{remark}
When  a symbol $a(x,\xi) $ is quantified in the semi-classical sense, the usual symbol is $a_h(x,\xi):=a(x,h\xi) $. 
\newline If $a\in S(m,g)$ then $a_h\in S(m_h,g_h)$ where $m_h(x,\xi):=m(x,h\xi)$ and $(g_h)_{(x,\xi)}(dx,d\xi):=g_{(x,h\xi)}(dx,hd\xi)$. 
\newline \noindent We verify that $$(g_h)^\sigma=h^{-2}(g^\sigma)_h$$
and the calculus is admissible if $g_h$ is temperate. 
\par In this case we say that {\it $g$ is semi-classical temperate}. The condition on $g$ read, there exists $C>0$ and $N>0$ such that for all $x,y,z,\xi,\eta,\zeta$,
$$
 g_{(x,\xi)}(z,\zeta)\le Cg_{(y,\eta)}(z,\zeta)(1+h^{-2}g^\sigma_{(x,\xi)}(x-y,\xi-\eta))^N
$$
Analogously, we say that {\it $m$ is semi-classical $\sigma,g$-temperate} if, there exists $C>0$ and $N>0$ such that, for all $x,y,\xi,\eta$,
$$
 m(x,\xi)\le Cm(y,\eta)(1+h^{-2}g^\sigma_{(x,\xi)}(x-y,\xi-\eta))^N
$$
On the other hand, if $\displaystyle{\mathfrak{h}^2_h(x,\xi):= \sup_{(y,\eta)} \left(\frac{(g_h)_{(x,\xi)}(y,\eta)}{(g_h)^\sigma_{(x,\xi)}(y,\eta)}\right)}$, one has
$$\mathfrak{h}_h(x,\xi)=h \mathfrak{h}(x,h\xi).$$
\end{remark}
\vspace{0.8cm}
\begin{proof}
 We have $g^\sigma=\este{\xi'}^2dx_1^2+\este{\xi''}^2dx''^2+d\xi'^2$.
If 
$$\frac{|\xi_1-\eta_1|^2}{\este{\xi'}^2}+\frac{|\xi''-\eta''|^2}{\este{\xi''}^2}<\delta$$
where $\delta$ is small enough, then $$\este{\xi'}^2\sim\este{\eta'}^2, \ \este{\xi''}^2\sim\este{\eta''}^2 \text{ and } \est{\xi'}^2\sim\est{\eta'}^2$$
uniformly with respect to $\eps$. We deduce that $g$ is slowly varying and that $\este{\xi'}$, $\este{\xi''}$ and $\est{\xi'}$ are $g$ continuous. 
\par To prove the temperance, the key point is the following estimate: there exists $C>0$ and $N>0$ such that  
\begin{equation}\label{temperance}
\forall \xi',\eta'\in\R^{n-1}, \ \este{\eta''}^2\le C\este{\xi''}^2(1+h^{-2}|\xi'-\eta'|)^N \textrm{ and } \este{\eta'}^2\le C\este{\xi'}^2(1+h^{-2}|\xi'-\eta'|)^N.
\end{equation}
Indeed, one has $$|\eta|^2+\eps^2\le 2(|\xi|^2+\eps^2)+2|\xi-\eta|^2\le 2 \este{\xi}^2(1+h^{-2}|\xi-\eta|^2)$$ if $ h\le \eps$. The temperance is then
a straightforward consequence of \eqref{temperance}.
\par Moreover, it is also easy to see that
$$
 \mathfrak{h}^2(x',\xi')=\sup_{(y',\eta')}\left(\frac{|y'|^2+\frac{\eta_1^2}{\este{\xi'}^2}+\frac{|\eta''|^2}{\este{\xi''}^2}}{\este{\xi'}^2y_1^2+\este{\xi''}^2|y''|^2+|\eta'|^2}\right)=\frac{1}{\este{\xi''}^2}.
$$
since $\este{\xi''}\le \este{\xi'}$ and equality is obtained for $y_1=1$, $y''=0$ and $\eta'=0$.

\end{proof}

In the sequel, we will need the following version of the sharp semi-classical G\aa{}rding inequality.

\begin{proposition}\label{Garding}
  Let $a \in S(1,g)$ such that $a\ge 0$.
  Then there exists $C_\eps>0$ and $h_\eps>0$ such that
  \[ \forall f \in L^2(\R^{n-1}), \forall h \in (0, h_\eps) ;
     \Re(\op (a) f| f)_0 \geqslant -C_\eps h | f |^2 , \]
where $(.,.)_0$  is the standard scalar product on $L^2(\R^{n-1})$.
\end{proposition}
\begin{proof}
We first note that
$$ \mathfrak{h}_h(x',\xi')=\frac{h}{ \este{h\xi''}}\le \frac{h}{\eps}$$
and, since $a\in S(1,g)$, 
$$ \frac{\eps}{h} a_h\in S\left(\frac{1}{\mathfrak{h}_h},g_h\right).$$
\newline \noindent If $h\le \eps$ then $\mathfrak{h}_h\le 1$ and the standard sharp G\aa{}rding inequality (see \cite[Theorem 18.6.7]{Hoermander:V3}) applied to $\frac{\eps}{h} a_h$ directly
gives
$$\forall f \in L^2(\R^{n-1}), \forall h \in (0, \eps) ;
     (\op_W (a) f| f)_0 \geqslant -C_\eps h | f |^2 ,$$
where $\op_W$ denotes the Weyl quantified operator associated to $a$.
\newline \noindent Moreover, one may classicaly write
$$ \op_W (a)=\op (\tilde{a})$$
for some $\tilde{a}$ which satisfies 
$$\tilde{a}=a+ h b$$
where $ b\in S(1,g)$.
The result is then a straightforward consequence of the $L^2$ continuity of $\op(b)$.
\end{proof}

\begin{lemma}\label{lemme estimation norme} 
 Let $(a_\eps)_{\eps\in(0,1)}$ a family of $S(1,g)$ and $$M_\eps=\sup_{(x',\xi')\in\R^{n-1}\times\R^{n-1}}|a_\eps(x',\xi')|.$$
Then 
there exists $C_\eps>0$ such that
\begin{equation}\label{lemme estimation norme eq 0}
\forall u\in \S(\R^{n-1}), \ |\op(a_\eps)u|\le (M_\eps+C_\eps h^{1/2})|u|
\end{equation}
provided $h$ is sufficiently small.
\end{lemma}

\begin{proof}The method of proof is classical. We give here a proof in our context.

By the symbolic calculus, we have  
$$ \op(a_\eps)^*\op(a_\eps)=\op(|a_\eps|^2)+\op(c)$$
where $c\in hS(1,g)$  and
\begin{equation}\label{lemme estimation norme, eq 1}
\forall u\in \S(\R^{n-1}), \ |(\op(c)u|u)_0|\le C_\eps h|u|^2.
\end{equation}
By assumption  $M_\eps^2-|a_\eps|^2\ge 0$ and $M_\eps^2-|a_\eps|^2\in S(1,g)$. Proposition \ref{Garding} consequently gives some $C_\eps>0$ such that 
$$
\forall u\in \S(\R^{n-1}), \ M_\eps^2|u|^2-(\op(|a_\eps|^2)u|u)_0+C_\eps h|u|^2\ge 0
$$
thus by \eqref{lemme estimation norme, eq 1} we have also
$$
 M_\eps^2|u|^2-|\op(a_\eps)u|^2+C_\eps h|u|^2\ge 0
$$
which gives \eqref{lemme estimation norme eq 0}.

\end{proof}

\begin{lemma}\label{estimation symbole nul en 0}
 Let $(a_\eps)_{\eps\in(0,1)}$ a family of $S(1,g)$. We assume that there exists $C>0$ such that 
$$\forall \eps\in (0,1),\ \forall (x',\xi'), a_{\eps}(0,\xi')=0  \text{ and } \ |\partial_{x'}a_\eps(x',\xi')|\le \frac{C}{\eps}.$$
Then there exists $C_1>0$ and $h_\eps, K_\eps>0$ such that for all $v_0\in H^{1/2}(\R^{n-1}) $ supported in $B_\kappa:=\{x'\in \R^{n-1},\ |x'|\le \kappa \}$,
\begin{equation}\label{estimation symbole nul en 0, eq : 1}
\forall h\le h_\eps, \ |\op(a_\eps)z_0|\le C_1\frac{\kappa}{\eps}|z_0|+K_\eps h^{1/2}|v_0|_{1/2}
\end{equation}
 with $z_0=\op(\lambda_+^{1/2})v_0$.
\end{lemma}

\begin{proof}

Let $\chi$ and $\tilde\chi$ in $\Cinf(\R^{n-1})$ such that $\chi=\tilde\chi=1$ on $B_\kappa$, supported in $B_{2\kappa}$ and $\supp \chi\subset \{\tilde\chi=1\}$.
We have $\lambda_+^{1/2}\in S(\este{\xi'}^{1/2}, g)$ hence, by symbolic calculus,

\begin{align}
 \op(a_\eps)z_0&=\op(a_\eps)\tilde\chi z_0+
\op(a_\eps)(1-\tilde\chi)\op(\lambda_+^{1/2})\chi v_0\nonumber\\
&=\op(a_\eps\tilde\chi) z_0+\op(a_1) z_0+\op(a_2) v_0,\label{estimation symbole nul en 0, eq : 2}
\end{align}
 where $a_1\in h S(1,g)$ and $a_2\in h S(\este{\xi'}^{1/2},g)$ (since the asymptotic expansion of $a_2$ is null).

\noindent We have, since $a_1\in h S(1,g)$,

\begin{equation}\label{estimation symbole nul en 0, eq : 3}
|\op(a_1)z_0|\le C_\eps h|z_0|.
\end{equation}
One has the estimate
\begin{equation}\label{estimation symbole nul en 0, eq : 5}
 \sup_{(x',\xi')}|a_\eps(x',\xi')\chi(x')|=
\sup_{(x',\xi')}\left|\chi(x')x'\int_0^1\partial_{x'}a_\eps(tx',\xi')dt\right|\le 2C\frac{\kappa}{\eps}
\end{equation}
and $a_\eps\chi\in S(1,g)$. We can apply  Lemma \ref{lemme estimation norme} and get, by  \eqref{estimation symbole nul en 0, eq : 5},
\begin{equation}\label{estimation symbole nul en 0, eq : 6}
 |\op(a_\eps\chi)z_0|\le (2C \kappa/\eps+C_\eps h^{1/2})|z_0|.
\end{equation}
On the oher hand, by symbolic calculus, we have 
\begin{equation*}
\op(\est{\xi'}^{-1/2})\op(a_2)=\op(a_3)
\end{equation*}
where $a_3\in  h S(\este{\xi'}^{1/2}\est{\xi'}^{-1/2},g)\subset h S(1,g)$. Consequently, 
\begin{equation}\label{estimation symbole nul en 0, eq : 7}
 |\op(a_2)v_0|=|\op(a_3)v_0|_{1/2}\le C_\eps h|v_0|_{1/2}.
\end{equation}
Finally, we have
\begin{equation}\label{estimation symbole nul en 0, eq : 8}
 |z_0|=|\op(\lambda_+^{1/2})v_0|\le C_\eps|v_0|_{1/2}.
\end{equation}
Following \eqref{estimation symbole nul en 0, eq : 2}, \eqref{estimation symbole nul en 0, eq : 3}, \eqref{estimation symbole nul en 0, eq : 6}, \eqref{estimation symbole nul en 0, eq : 7} and \eqref{estimation symbole nul en 0, eq : 8}, 
we get the estimate \eqref{estimation symbole nul en 0, eq : 1}.
\end{proof}

\section{Symbol reduction and roots properties}

We recall that
$$p_{\VP}(x,\xi)=\xi_n^2+2i\partial_{x_n}\VP\xi_n+q_2(x,\xi')+2iq_1(x,\xi')$$
where
$q_2(x,\xi')=-(\partial_{x_n}\VP(x))^2+r(x,\xi')-r(x,\partial_{x'} \VP(x))$, $q_1(x,\xi')=\tilde r(x,\xi',\partial_{x'} \VP(x))$
and $$\mu(x,\xi')=q_2(x,\xi')+\frac{q_1(x,\xi')^2}{(\partial_{x_n}\VP(x))^2}.$$
\begin{lemma}\label{Lemme sur les racines}
 $ $
\begin{itemize}
 \item
\begin{itemize}
 \item If $\mu(x,\xi')<0$, the roots of of $p_{\VP}(x,\xi',\xi_n)$ with respect to $\xi_n$ have negative imaginary parts.
 \item On the other hand, if $\mu(x,\xi')>-(\partial_{x_n}\VP(x))^2$ the two roots of $p(x,\xi',\xi_n)$ with respect to $\xi_n$ have different imaginary parts. 

If we denote by $\rho_1(x,\xi')$ and $\rho_2(x,\xi')$ the roots such that $\Im (\rho_1(x,\xi'))>\Im (\rho_2(x,\xi'))$ then
\begin{equation*}
\exists C>0, \ \de>0\text{ such that } |\xi'|\ge C \Rightarrow \Im (\rho_1(x,\xi'))\ge \de |\xi'|,
\end{equation*}
\begin{equation}\label{racines separee}
 \Im (\rho_1(x,\xi'))>-\partial_{x_n}\VP(x) >\Im (\rho_2(x,\xi')).
\end{equation}
Moreover, if $\alpha>0$  there exists $\delta>0$ such that for all $(x,\xi')\in \R^{n-1} \times \R^{n-1}$ satisfying $\mu(x,\xi')\ge-(1-\alpha)(\partial_{x_n}\VP(x))^2$,
\begin{equation}\label{racines elliptiques}
 |\rho_1(x,\xi')+i\partial_{x_n}\VP(x)|\ge\delta \est{\xi'}  \textrm{ and }   |\Im\rho_2(x,\xi')|\ge\delta\est{\xi'}.
\end{equation}
\end{itemize}
\item  If $r(0,\xi')=|\xi'|^2$, $|\partial_{x'}\VP(0)|\le \eps \partial_{x_n}\VP(0)$ and, for some $\alpha>0$, $\mu(0,\xi')\ge -(1-\alpha) (\partial_{x_n}\VP(0))^2$  
then, there exists $\delta>0$ such that
\begin{equation}\label{racines x=0,3}
|\xi'|\ge \delta \partial_{x_n}\VP(0) 
\end{equation}
and, for sufficiently small $\eps$, the following formula is valid 
\begin{equation}\label{racines x=0,1}
\rho_1(0,\xi')+i\partial_{x_n}\VP(0)=i \partial_{x_n}\VP(0) \rho\left(\frac{\xi'}{\partial_{x_n}\VP(0)}\right)
\end{equation}
where, setting $V:=\frac{\partial_{x'}\VP(0)}{\partial_{x_n}\VP(0)}$, we have noted
$$\rho(\eta'):=\sqrt{|\eta'|^2-|V|^2+2i\eta'.V} \ .$$ 
Moreover, one has the following estimates
\begin{eqnarray}\label{racines x=0,2}
&\forall |\eta'|\ge \eps,& \quad |\eta'|+\eps \ge \Re(\rho(\eta'))\ge  |\eta'| -\eps,\\
&\forall |\eta'|\ge 2\eps,& \quad |\Im(\rho(\eta'))|\le 2\eps. \nonumber
\end{eqnarray}
\end{itemize}
\end{lemma}

\begin{proof}
The first point is proved in \cite[Lemme 3]{LR2} using some geometric transformation. For the reader's convenience, we however give another elementary proof.
\\For simplicity, we do not write the variables $(x,\xi')$ and we define the two roots $\rho_1$ and $\rho_2$ such that $\Im \rho_1 \ge \Im \rho_2$. Using equation $p_\VP=0$, we have $$\rho_1+\rho_2=-2i\partial_{x_n}\VP$$ and consequently, 
there exists $a,b \in \R$ such that
$$ \rho_1=a+ib, \quad \rho_2=-a -i (2 \partial_{x_n} \varphi+b).$$ 
It is sufficient to show that $b<0$. Moreover, since $\rho_1 \rho_2= q_2+2i q_1$ and $q_1, q_2$ are real-valued, one gets
$$q_2=-a^2+b(2\partial_{x_n}\varphi+b), \quad q_1=-a(2\partial_{x_n}\varphi+b)$$
and consequently, if $b \ge 0$,
$$q_2+ \frac{q_1^2}{(\partial_{x_n} \varphi)^2}=a^2 \left( \left(2+ \frac{b}{\partial_{x_n} \varphi}\right)^2-1\right)+b(2\partial_{x_n}\varphi+b) \ge 0;$$
a contradiction with our assumption.
\\We now prove that $\rho_1$ and $\rho_2$ have different imaginary parts if $\mu(x,\xi')>-(\partial_{x_n}\VP(x))^2$.  
Assuming that $\Im \rho_1=\Im \rho_2$, by equation $p_\VP=0$ there also exists $a\in\R$ such that $$\rho_1=a-i\partial_{x_n}\VP \text{ and } \rho_2=-a-i\partial_{x_n}\VP.$$
Thus $q_2+2iq_1=\rho_1\rho_2=-(\partial_{x_n}\VP)^2-a^2$ hence $q_1=0$ and $\mu=-(\partial_{x_n}\VP)^2-a^2\le -(\partial_{x_n}\VP)^2$; a contradiction with the assumption.
\newline \noindent Consequently, since $\rho_1+\rho_2=-2i\partial_{x_n}\VP$ and $\Im \rho_1>\Im \rho_2$, we obtain \eqref{racines separee}.

Moreover, we have $$q_2(x,\xi')=r(x,\xi')+O(1) \text{ and } q_1(x,\xi')=O(\est{\xi'}).$$ It is easy to deduce that $$\rho_1(x,\xi')=i\sqrt{r(x,\xi')}+O(1) \text{ and } \rho_2(x,\xi')=-i\sqrt{r(x,\xi')}+O(1).$$
Hence for $|\xi'|>C$ where $C$ is large enough, we have $$\Im(\rho_1)\ge \de|\xi'|, \ |\rho_1(x,\xi')+i\partial_{x_n}\VP(x)|\ge \delta|\xi'| \text{ and } |\Im\rho_2(x,\xi')|\ge \delta|\xi'|$$
where $\delta>0$ is sufficiently small. But we have already proved that $\rho_1(x,\xi')+i\partial_{x_n}\VP(x)\not=0$ and $\Im\rho_2(x,\xi')\not=0$ thus, by compactness argument on $|\xi'|\le C$, we get \eqref{racines elliptiques}.

\par If $r(0,\xi')=|\xi'|^2$ then 
\begin{eqnarray*}p_{\VP}(0,\xi',\xi_n)&=&(\xi_n+i\partial_{x_n}\VP(0))^2+(\xi'+i\partial_{x'}\VP(0)).(\xi'+i\partial_{x'}\VP(0))\\
 &=& (\xi_n+i\partial_{x_n}\VP(0))^2+|\xi'|^2-|\partial_{x'}\VP(0)|^2+2i\xi'.\partial_{x'}\VP(0).
\end{eqnarray*}
On the other hand, if $\mu(0,\xi')\ge -(1-\alpha) (\partial_{x_n}\VP(0))^2$, one has, writing $\eta'=\frac{\xi'}{\partial_{x_n}\VP(0)}$,
$$|\eta'|^2-|V|^2\ge |\eta'.V|^2+\alpha \quad \text{ so that } \quad |\eta'|\ge \sqrt{\alpha}.$$
This proves \eqref{racines x=0,3}.
\newline \noindent Hence, for sufficiently small $\eps$ and since $|\partial_{x'}\VP(0)|\le \eps \partial_{x_n}\VP(0)$, we get 
$$|\xi'|^2-|\partial_{x'}\VP(0)|^2+2i\xi'.\partial_{x'}\VP(0) \not\in \R_-$$
and we consequently deduce
$$\rho_1(0,\xi')+i \partial_{x_n}\VP(0)=i\sqrt{|\xi'|^2-|\partial_{x'}\VP(0)|^2+2i\xi'.\partial_{x'}\VP(0)},$$
which immediately give \eqref{racines x=0,1}.
\newline Using that, for any $z\in \C$ with $\Re(z)\ge 0$,
$$\Re(\sqrt{z})\ge \sqrt{\Re(z)} \text{ and } 2\Re(\sqrt{z})\Im(\sqrt{z})=\Im(z),$$
we finally have, for $|\eta'|\ge \eps$ and since $|V|\le \eps$,
$$\Re(\rho(\eta'))\ge \sqrt{|\eta'|^2-|V|^2}\ge |\eta'|-|V|\ge |\eta'|-\eps$$
and, for $|\eta'|\ge 2 \eps$,
$$|\Im(\rho(\eta'))|\le \frac{|V| |\eta'|}{|\eta'|-\eps} \le \eps \frac{|\eta'|}{ |\eta'|-1/2|\eta'|}\le 2 \eps.$$
On the other hand, one has
$$|\Re(\rho(\eta'))|\le \left( (|\eta'|^2-|V|^2)^2+4(\eta'.V)^2\right)^{1/4}\le (|\eta'|^2+|V|^2)^{1/2} \le |\eta'|+|V|$$
which concludes the proof since $|V|\le \eps$.
\end{proof}

\begin{lemma}\label{Lemme forme normale}
Let $p(x,\xi)$ a $\Cinf$ positive definite quadratic form in $T^*(U)$ where $U$ is a neighborhood of $0\in \R^n$. 
\newline Let $f$ a smooth function defined in $U$ satisfying $f(0)=0$ and $df\neq 0$. 
\newline Then there exists a change of variables such that, in the new variables $(y,\eta)$, we have locally near 0
\begin{description}
\item[] $f(x)>0 \Leftrightarrow y_n>0$ and $f(x)=0 \Leftrightarrow y_n=0$,
\item $q(y,\eta)=\eta_n^2+r(y,\eta')  \textrm{ if }\eta=(\eta',\eta_n)$,
 \end{description}
with $q$ the symbol $p$ written in the new variables.
\newline Moreover, if $k$ a smooth function defined on a \nhd of $0$ in the submanifold $S:=f^{-1}(\{0\})$ such that $k(0)=0$ and $dk(0)\in T_0^* S\setminus \{0\}$ 
we can choose the new variables $(y',\eta')$ such that
\begin{description}
\item $ (k(x)>0\textrm{ on } f(x)=0\Leftrightarrow y_1>0 \textrm{ on }y_n=0)$ and $(k(x)=0\textrm{ on } f(x)=0\Leftrightarrow y_1=0 \textrm{ on }y_n=0)$,
\item $r(0,\eta')=|\eta'|^2$.
 \end{description}
\end{lemma}
\begin{proof}
It is classical (see \cite[Corollary C.5.3]{Hoermander:V3} for instance) that we can find change of variables such that 
$$(f(x)>0 \Leftrightarrow y_n>0), \quad (f(x)=0 \Leftrightarrow y_n=0)\text{ and } q(y,\eta)= \eta_n^2+r(y,\eta')$$
 with $r$ a homogeneous polynomial of order 2 in $\eta'$.
Moreover, using Taylor formula, one may write $$r(y,\eta')=r(y',0,\eta')+y_n r_1(y,\eta').$$
By the same method, we can choose variables $z'=(z_1,z'')$ on $S$ such that, denoting $s$ the function $r$ written in the variables $(z',\zeta')$,
$$ (k(y')>0\Leftrightarrow z_1>0),\quad (k(y')=0\Leftrightarrow z_1=0)\text{ and } s(z',0,\zeta')=\zeta_1^2+s_1(z',\zeta''),$$
where we have set $\zeta'=(\zeta_1,\zeta'')$. 
\newline Finally, by a linear change of variables in $z''$ (which does not pertub the other term), we can write $s_1(0,\zeta'')=|\zeta''|^2$ and consequently
get
$$s(0,\zeta')=|\zeta'|^2.$$
\end{proof}

\begin{remark}
$ $\newline 
\begin{itemize}
 \item A $\Cinf$ positive definite quadratic form in $T^*(U)$ is a $\Cinf$ map such that for all $x\in U$, $p(x,.)$ is a positive definite quadratic form.
\item We note that this Lemma takes a more standard form if $k$ can be defined in a \nhd of $0\in \R^n$. Indeed, one has the following result. 
\par
\par \noindent If $f$ and $k$ are smooth function defined in a neighborhood of $0$, satisfying $f(0)=k(0)=0$ and such that $df(0)$ and $dk(0)$ are independent, then there exists
a change of variables such that, in the new variable $(y,\eta)$ and locally near $0$, 
\begin{description}
 \item[]$f(x)>0 \Leftrightarrow y_n>0$,
\item $ k(x)>0\textrm{ on } f(x)=0\Leftrightarrow y_1>0 \textrm{ on }y_n=0$,
\item $q(y,\eta)=\eta_n^2+r(y,\eta')  \textrm{ where }\eta=(\eta',\eta_n)$,
\item $r(0,\eta')=|\eta'|^2$.
 \end{description}
\end{itemize}
\end{remark}

\section{A norm estimate}
\label{Appendix C}
\begin{lemma}\label{holderpetit}
Let $\rho \in (0,1)$, $a\in \Con^\rho(\R^{n-1})$ such that $a_{|x_1=0}=0$ and $\chi$ a smooth function supported in the unit ball $B_1$. Then there exists $C>0$ such that, for any $\rho'\le \rho$ and any $\lambda\in(0,1)$,
$$\|a \chi(./\lambda)\|_{\Con^{\rho'}} \le C \lambda^{\rho-\rho'} \|a\|_{\Con^{\rho}}$$
where, for any $x\in \R^{n-1}$, $\chi(./\lambda)(x)=\chi(x/\lambda).$
\end{lemma}
\begin{proof}
We first estimate $|(a\chi(./\lambda))(x)-(a\chi(./\lambda))(y)|$ depending on the positions of $x$ and $y$.
\begin{itemize}
\item If $x, y \notin B_\lambda$, one clearly has
$$|(a\chi(./\lambda))(x)-(a\chi(./\lambda))(y)|\le C \lambda ^{\rho-\rho'}  \|a\|_{\Con^\rho} \ |x-y|^{\rho'}.$$
\item If $x, y \in B_\lambda$, one has
$$
(a\chi(./\lambda))(x)-(a\chi(./\lambda))(y)=a(x)\left( \chi\left(\frac{x}{\lambda}\right)-\chi\left(\frac{y}{\lambda}\right)\right)+(a(x)-a(y)) \chi\left(\frac{y}{\lambda}\right) 
$$
which gives, since $|x-y|\le 2 \lambda$ and $\chi$ is smooth, 
\begin{eqnarray*}
|(a\chi(./\lambda))(x)-(a\chi(./\lambda))(y)| &\le& C \left(|a(x)|\frac{|x-y|}{\lambda} +\|a\|_{\Con^\rho} |x-y|^\rho\right)\\
& \le& C \left(|a(x)| \frac{|x-y|^{\rho'}}{\lambda^{\rho'}}  +\lambda ^{\rho-\rho'} \|a\|_{\Con^\rho}  |x-y|^{\rho'}\right).
\end{eqnarray*}
Using now that $a_{|x_1=0}=0$, one has moreover
$$|a(x)| \le |x_1|^\rho \|a\|_{\Con\rho} \  \le C \lambda^\rho\ \|a\|_{\Con^\rho}$$
and finally gets
$$\forall x, y \in B_\lambda, \  |(a\chi(./\lambda))(x)-(a\chi(./\lambda))(y)|\le C \lambda ^{\rho-\rho'}  \|a\|_{\Con^\rho} \ |x-y|^{\rho'}.$$
\item If now $x \in B_\lambda$, $y\notin B_\lambda$, one has 
$$|(a\chi(./\lambda))(x)-(a\chi(./\lambda))(y)|=|(a\chi(./\lambda))(x)|=|(a\chi(./\lambda))(x)-(a\chi(./\lambda))(t)|$$
where $t \in [x,y]$ and $|t|=\lambda$. Using the second case, we also get
$$|(a\chi(./\lambda))(x)-(a\chi(./\lambda))(y)| \le C \lambda ^{\rho-\rho'}  \|a\|_{\Con^\rho} \ |x-y|^{\rho'}$$
since $|x-t|\le |x-y|$.
\end{itemize}
Finally, it is obvious that
$$\|a \chi(./\lambda)\|_{L^\infty}\le \|\chi\|_{L^{\infty}} \|a\|_{L^\infty(B_\lambda)} \le C \|a\|_{L^\infty(B_\lambda)}$$
and since, for any $ x\in B_\lambda$, $ |a(x)| \le C \lambda ^{\rho}  \|a\|_{\Con^\rho} $,
$$\|a \chi(./\lambda)\|_{L^\infty}\le C \lambda^{\rho-\rho'}  \|a\|_{\Con^\rho}.$$
The proof is complete.
\end{proof}

\end{document}